\documentclass[11pt]{article}
\usepackage{amsmath}
\usepackage{amsthm}
\usepackage{amssymb}
\usepackage{geometry,mathtools}
\usepackage{url}

\vspace{0.2cm}
\newtheorem{theo}{Theorem}[section]

\newtheorem{lemma}[theo]{Lemma}
\newtheorem{coro}[theo]{Corollary}
\newtheorem{conj}[theo]{Conjecture}

\newtheorem{claim}[theo]{Claim}

\newcommand{\eps}{{\varepsilon}}
\newcommand{\vv}{\mathbf{v}}
\newcommand{\uu}{\mathbf{u}}
\newcommand{\yy}{\mathbf{y}}

\newcommand{\zz}{\mathbf{z}}

\makeatletter
\renewenvironment{proof}[1][\proofname]{\par
  \pushQED{\qed}%
  \normalfont \topsep6\p@\@plus6\p@\relax
  \noindent\textbf{#1}\@addpunct{.}\quad\ignorespaces
}{%
  \popQED\endtrivlist\@endpefalse
}
\makeatother

\begin{document}
\date{}

\title{
The limit points of the top and bottom eigenvalues of regular graphs}
%\\ {\small eiglimit1.tex} }

\author{Noga Alon
\thanks{Department of Mathematics, Princeton University,
Princeton, NJ 08544, USA and
Schools of Mathematics and
Computer Science, Tel Aviv University, Tel Aviv 6997801,
Israel.
Email: {\tt nalon@math.princeton.edu}.
Research supported in part by
NSF grant DMS-2154082 and USA-Israel BSF grant 2018267.}
\and
Fan Wei
\thanks{
Department of Mathematics, Duke University, Durhan, NC 27708.
Email: {\tt fan.wei@duke.edu}.
Research supported by NSF Award
DMS-1953958.}
}

\maketitle
\begin{abstract}
We prove that for each $d \geq 3$ 
the set of all limit points of the second largest
eigenvalue of growing sequences of $d$-regular graphs is
$[2 \sqrt{d-1},d]$. A similar argument shows that the  set of all
limit points of the smallest eigenvalue of growing sequences
of $d$-regular graphs with growing (odd) girth is 
$[-d, -2 \sqrt{d-1}]$.  The more general question of identifying
all vectors which are limit points of the vectors of the top
$k$ eigenvalues of sequences of $d$-regular graphs is considered
as well. As a by product, 
in the study of  discrete counterpart of the “scarring”
phenomenon observed in the investigation of quantum ergodicity on manifolds, 
our technique provides a method to construct $d$-regular 
almost Ramanujan graphs with large girth and localized 
eigenvectors corresponding to eigenvalues larger 
than $2\sqrt{d-1}$, strengthening a result of Alon, Ganguly, and Srivastava \cite{AGS}. 
\end{abstract}

\section{Introduction}
For a graph $G$ on $n$ vertices, let $\lambda_1(G) \geq \lambda_2(G)
\ldots \geq \lambda_n(G)$ denote the ordered set of its eigenvalues.
If $G$ is $d$-regular then $\lambda_1(G)=d$ and the Alon-Boppana bound
(\cite{Al}, \cite{Ni}, \cite{Ni1}) asserts that
$\lambda_2(G) \geq 2\sqrt{d-1}(1-O(1/\log^2 n))=(1-o(1))2\sqrt{d-1}$.
Therefore, any limit point of the values of $\lambda_2(G_i)$
for an infinite sequence $G_i$ of $d$-regular graphs is at least
$2\sqrt{d-1}$, and the existence of near-Ramanujan graphs for every degree
$d$ proved in \cite{Fr} (see also \cite{Bo}, \cite{MSS}) implies that
$2\sqrt{d-1}$ is such a limit point. Are all other points in
the interval $[2\sqrt{d-1},d]$ also obtained as such limit points?
This question was suggested to us by Peter Sarnak. 
It can be viewed as a variation of the  inverse spectral problem, whose analogue for hyperbolic surfaces 
 was recently proved using the conformal bootstrap method  \cite{KMP}. 

Our first result 
in this paper is a short proof that this is indeed the case.
\begin{theo}
\label{t11}
Let $A_2(d)$ denote the set of all limit points of sequences
$\lambda_2(G_i)$ where $G_i$ is an infinite sequence of $d$-regular
graphs, then for every $d \geq 3$, $A_2(d)=[2 \sqrt{d-1},d]$.
\end{theo}

Note that trivially not every value in $A_2(d)$ can be achieved 
as $\lambda_2(G)$ for some finite $d$-regular graph $G$, as such a 
value needs to be a totally real algebraic integer (and also as the number
of finite graphs is only countable). 
Indeed, most values in $A_2(d)$ 
can only be achieved as limit points of sequences 
$\lambda_2(G_i)$ as in the theorem.

Our method implies a similar result for the set of all limit points
of the smallest eigenvalue of growing sequences of $d$-regular graphs,
provided the (odd) girth of these graphs tends to infinity. 
\begin{theo}
\label{t12}
Let $A_{s}(d)$ denote the set of all limit points of sequences
of the last (smallest) eigenvalue of growing sequences of
$d$-regular graphs in which the length of the shortest odd cycle
tends to infinity.
For every $d \geq 2$, $A_s(d)=[-d,-2 \sqrt{d-1}]$.
\end{theo}
Without the assumption about the growing odd girth 
the set of limit points of the smallest eigenvalue 
is more complicated, contains isolated points, and is 
far from being fully understood. See \cite{Yu}
for the values of the first few largest points of this set.

We conjecture that the assertion of Theorem \ref{t11} can be extended
to determine the limit points of the vectors in $\mathbb{R}^k$ of the top $k$
eigenvalues, for any fixed $k$.
\begin{conj}
\label{c13}
For any $d \geq 3$  and any fixed $k$, the set of all limit points of 
the vectors
$$
(\lambda_1(G_i),\lambda_2(G_i), \ldots ,\lambda_k(G_i))
$$
for an infinite sequence  $G_i$ of $d$-regular graphs is exactly the
set 
$$
B(d,k)=\{(\mu_1,\mu_2, \ldots ,\mu_k):~d=\mu_1 \geq 
\mu_2 \geq \mu_3 \geq \ldots \geq \mu_k \geq 2\sqrt{d-1} \}.
$$
\end{conj}
The fact that the above set of limit points is contained 
in $B(d,k)$ follows from the known result, observed by several
researchers, c.f., e.g., \cite{Ci, Ni1}, that for any
fixed $d$ and $k$, any
sufficiently large $d$-regular graph has at least $k$
eigenvalues which are at least $2\sqrt{d-1}-o(1)$. However, we have
not been able to decide whether or not 
the set of these limit points contains every point of
$B(d,k)$ even for $k=3$. On the other hand, we can prove the following,
showing that every point of $B(d,k)$ can be obtained as a limit
if we relax the regularity condition.
\begin{theo}
\label{t14}
For every $d \geq 3$ and every $k$, every point of $B(d,k)$ is a limit
point of a sequence of vectors 
$(\lambda_1(G_i),\lambda_2(G_i), \ldots ,\lambda_k(G_i))$
for an infinite sequence  $G_i$ of graphs with maximum
degree at most $d$.
\end{theo}
Note that, of course, for sequences of 
graphs with maximum degree $d$ there are 
also limit points as above that lie outside the set $B(d,k)$.

For $d$-regular graphs we can prove that  Conjecture \ref{c13} is almost true, in the sense that 
every point of
$B(d,k) \cap [2\sqrt {d-1} +o_d(1),d]^k$ can be obtained as a limit point. To be more precise, we have the following theorem. 

\begin{theo}\label{thm:limit}
For every $d \geq 3$ and every $k$, every point of
$$
\{(\mu_1,\mu_2, \ldots ,\mu_k):~d=\mu_1 \geq 
\mu_2 \geq \mu_3 \geq \ldots \geq \mu_k \geq 2\sqrt{d-1} + \frac{1}{\sqrt{d-1}}\}.
$$ 
is  a limit
point of a sequence of vectors 
$(\lambda_1(G_i),\lambda_2(G_i), \ldots ,\lambda_k(G_i))$
for an infinite sequence  $G_i$ of $d$-regular graphs.
\end{theo}

A byproduct of the technique is the following theorem, which constructs $d$-regular almost Ramanujan graphs $G$ with large girth, while ensuring the presence of a localized eigenvector corresponding to an eigenvalue strictly greater than $2\sqrt{d-1}$. This in some way strengthens a result of Alon, Ganguly and 
Srivastava \cite{AGS}, who showed the existence of large girth $d$-regular graphs $G$ with $\lambda_2(G)\leq 2.121\sqrt{d}$ and localized eigenvectors with eigenvalues in $(-2\sqrt{d}, 2\sqrt{d})$. \footnote{However their localized eigenvectors are completely supported on a small set.} Earlier research on 
constructing $d$-regular graphs with localized eigenvectors can be found in 
\cite{GS}, but the graphs produced there are not expanding.

This line of research involving constructions of high girth (and expanding) 
graphs that exhibit localized eigenvectors can be viewed as a discrete version of the ``scarring" phenomenon observed 
in the study of quantum ergodicity on manifolds. 
More about this topic can be found in 
\cite{BL}, \cite{AM}, \cite{BML}, \cite{GS}, and the references therein.

\begin{theo}\label{thm:vectormain}
For any real numbers $\beta > 0$ and $C \geq 4$, and  any positive integer $d = p+1$ where $p \equiv 1 \mod 4$ is a prime, %and for any $0<\alpha < 1/9$, 
there are infinitely many $d$-regular  graphs $G$ satisfying the following properties simultaneously:
\begin{enumerate}
    \item The second largest eigenvalue  is at most $ 2\sqrt{d-1} + \beta$,
    \item The girth is at least
    $\frac{1}{5C}\log_d |V(G)|$, and 
    \item There is a vertex set $S$ with $|S| \leq |V(G)|^{2/\sqrt{C}}$ 
and an  eigenvalue strictly larger than $2\sqrt{d-1}$ whose 
eigenvector $\vv$ satisfies $\sum_{u \in S} \vv(u)^2 
\geq (1-\beta)\|\vv\|_2^2$. 
\end{enumerate} 
\end{theo}
Our construction of these graphs is explicit. Note that we do not make any 
serious attempt to optimize the constants in the bounds for the girth
and for the exponent of $|V(G)|$ in the bound for $|S|$.
With a slightly more careful analysis, it is possible
to prove $\sum_{u \in S} \vv(u)^2 \geq (1-o(1))\|\vv\|_2^2$ where the 
$o(1)$ term here tends to $0$ as $|V(G)|$ tends to infinity 
(where $\beta,C,d$ are fixed).
It is also worth noting that a similar result also holds for 
any $d \geq 3$, with the bound on the girth being 
$\Omega(\log \log |V(G)|)$. The construction can again be made explicit by 
a deterministic $\text{poly}(|V(G)|)$-time algorithm, combining our
arguments with the construction in \cite{MDP}. 

\iffalse
The next theorem constructs almost Ramanujan $d$-regular graphs with large 
girth and multiple localized eigenvectors corresponding to eigenvalues 
strictly larger than $2\sqrt{d-1}$.     
\begin{theo}\label{thm:k}
Fix an  integer $k\geq 1$. Let $\beta > 0$ and $\epsilon>0$. For any sufficiently large  integer $d = p+1$ where $p \equiv 1 \mod 4$ is a prime, there are infinite number of $d$-regular  graphs satisfying the following properties simultaneously:
\begin{enumerate}
    \item its second largest eigenvalue  is at most $ 2(1+\beta)\sqrt{d-1}$,
    \item its girth is at least
    $\Omega(\log |V(F)| / \log d)$, %$\frac{1}{(4c+2)^2} \log |V(F)| / \log d$
    \item there are $k$ disjoint  vertex sets $S_i$ and $k$ eigenvectors $\vv_i$  for $1 \leq i \leq k$  such that  $|S_i| = O(\sqrt{|V(G)|/\log |V(G)|})$ and  $\sum_{u \in S_i} \vv_i(u)^2 \geq (1-\epsilon)\|\vv_i\|_2^2$. Furthermore, these eigenvectors correspond to  eigenvalues larger than $2\sqrt{d-1}$.
\end{enumerate} 
\end{theo}
\fi

The rest of this paper is organized as follows. 
In the next section we describe
the proof of the basic results: Theorem \ref{t11} and Theorem \ref{t12}.
Theorem \ref{t14} and several extensions are proved in Section 3. 
In Section 4  we describe the proof of Theorem \ref{thm:limit}.
Near Ramanujan graphs with localized eigenvectors are
constructed in Section  5 and the final Section 6 contains some 
concluding remarks. 

\section{Proofs of the basic results}

\subsection{The second eigenvalue}
We start with the proof of
Theorem \ref{t11} in the following stronger form.
\begin{theo}
\label{t21}
For every $d \geq 2$, every even integer $n  > d$  and every 
real $\lambda \in [2\sqrt{d-1},d]$ there is a $d$-regular graph $G$ 
on $n$ vertices satisfying  $|\lambda_2(G)-\lambda| \leq 
\frac{c}{\log \log n}$, where $c=c(d)>0$.
\end{theo}

\begin{proof}
The proof is based on the fact that eigenvectors of high-girth
$d$-regular graphs are nonlocalized. This enables us to start
from a near Ramanujan $d$-regular 
graph of high girth and apply to it local changes
(swaps),
transforming it to a graph with a very small bisection width
which has second eigenvalue close to $d$. This is done while maintaining
the high girth. The nonlocalized nature of the eigenvectors is used
to show that in each swap the second eigenvalue can change only by
a small amount. A more precise description follows.
 
We prove that there are constants 
$c_1,c_2,c_3,c_4>0$ (depending on $d$) and 
a (finite) sequence
of $d$-regular graphs $G_0,G_1, \ldots ,G_t$, each being a graph
on the set $V=\{v_1,v_2, \ldots ,v_n\}$ of $n$ labeled vertices,
that satisfy the following properties.
\begin{enumerate}
\item
$\lambda_2(G_0) \leq 2\sqrt{d-1} +c_1(\frac{\log \log n}{\log n})^2.$
\item
$\lambda_2(G_t) \geq d-\frac{c_2}{\sqrt n}.$
\item
The girth of each $G_i$ is at least $c_3 \log \log n$.
\item
For every $0 \leq i <t$, if either $\lambda_2(G_i)$ or $\lambda_2(G_{i+1})$
exceeds $2 \sqrt{d-1}$ then
$$
|\lambda_2(G_{i+1})-\lambda_2(G_i)| \leq \frac{c_4}{\log \log n}.
$$
\end{enumerate}
This clearly implies the assertion of Theorem \ref{t21}.  The existence
of $G_0$ satisfying properties (1) and (3) follows from the result of
Friedman \cite{Fr} and the known facts about the girth of random
regular graphs. Indeed, Friedman proved that for every $0<a<1$
there is some $c_1=c_1(a,d)$ so that
with probability at least
$1-n^{-a}$ a random $d$-regular graph $G_0$ on $n$ vertices satisfies
(1). By the known results about the 
distribution of short cycles in random $d$-regular graphs
(see \cite{MWW}) the girth of $G_0$ exceeds some 
$c_3 \log \log n$ with probability exceeding $1/\sqrt n$.
Taking $a=1/2$ in Friedman's result it follows that with positive
probability $G_0$ satisfies (1) and (3). Fix such a graph 
$G_0$.

For every $i \geq 0$, the graph $G_{i+1}$ will be constructed 
from $G_i$ by a single {\em swap}, that is, by deleting two
vertex disjoint edges $v_1v_2$ and $u_1u_2$
and by adding the edges $v_1u_1$ and $v_2u_2$ instead, keeping
the graph $d$-regular. We need the following result.
\begin{claim}
\label{cl22}
Let $G=(V,E)$ be a graph with maximum degree $d$ and with girth at least
$2r$, and let $x$ be an eigenvector of $G$ with 
$\ell_2$ norm $\|x\|_2=1$
corresponding 
to an eigenvalue $\mu$ with absolute value at least
$2 \sqrt{d-1}$. Then $\|x\|_{\infty} \leq 1/\sqrt r.$
\end{claim}
This is an immediate consequence of Lemma 3.2 of
\cite{Al1} where it is shown that if $uv$ is any edge 
in  such a graph, and $N_i$ is the set of all vertices of distance
exactly $i$ from $N_0=\{u,v\}$, $(0 \leq i \leq r-1)$, then for 
every $0<i \leq r-1$, $\sum_{v \in N_i} x_v^2 \geq \sum_{u \in N_{i-1}}
x_u^2$. Thus $x_u^2+x_v^2 \leq 1/r$. A slight extension of this lemma
for bounded degree (not necessarily regular) 
graphs is proved as Lemma \ref{lem:treemiddlelayer} below.
Note that a similar statement for regular graphs with
a slightly worse quantitative estimate can be derived from the
results in \cite{GS}, even without any assumption on the eigenvalue
$\mu$. There are, however, simple examples showing that this more
general statement does not hold for general bounded degree graphs.

\begin{coro}
\label{c23}
If $G$ and $H$ are two $d$-regular graphs, each
having girth at least $2r$, and one of them is obtained from the other
by a swap, and if either $\lambda_2(G)$ or $\lambda_2(H)$ exceeds
$2 \sqrt{d-1}$, then $|\lambda_2(G)-\lambda_2(H)| \leq 8/r$.
\end{coro}
Indeed, without loss of generality $\lambda_2(G) \geq \lambda_2(H)$
and $\lambda_2(G) \geq 2\sqrt{d-1}$. Let $\vv$ be a normalized eigenvector
of $\lambda_2(G)$. Then it is orthogonal to the constant vector,
and $\vv^T A_G \vv =\lambda_2(G)$, where $A_G$ is the adjacency matrix
of $G$. Since $H$ is obtained from $G$ by a single swap,
$\vv^T A_G \vv-\vv^T A_H \vv$ is a sum and difference of at most $8$ terms 
of the form $\vv(u)\vv(v)$. Here $A_H$ is the adjacency matrix of $H$. Since $\|\vv\|_{\infty} \leq 1/\sqrt r$ each
such term has absolute value at most $1/r$.  It follows that
$\vv^T A_H \vv \geq \lambda_2(G)-8/r$ and by the variational 
definition of $\lambda_2(H)$ this implies that
$\lambda_2(H) \geq \lambda_2(G)-8/r$, as needed.

Starting with $G_0$ satisfying (1) and (3) split its vertex
set arbitrarily into two sets of equal cardinality $B$ and $C$.
Suppose we have already defined $G_0,G_1, \ldots ,G_i$ so that
(3) and (4) hold, and suppose that $G_i$ still has more than
$\sqrt n$ crossing edges, that is,
edges with endpoints in $B$ and in $C$. We show how to define
$G_{i+1}$ and decrease the number of these crossing edges by $2$.
Let $v_1v_2$ be an arbitrary crossing edge of $G_i$, with
$v_1 \in B, v_2 \in C$. The number of
edges of $G_i$ whose distance from the edge $v_1v_2$ is at most
$2r-1$ is at most $1+2(d-1)+2(d-1)^2 + \cdots +2 (d-1)^{2r-1}
<2d^{2r}$. If $r$ is smaller than $(1/4)\log(n/4)/\log d $ this number is
smaller than $\sqrt n$, and hence there is at least one additional
crossing edge $u_1u_2$ of $G_i$ with $u_1 \in B,u_2 \in C$.
Let $G_{i+1}$ be the graph obtained from $G_i$ by the swap that
removes the edges $v_1v_2, u_1u_2$ and adds the edges
$u_1v_1$, $u_2v_2$. Since any cycle of $G_{i+1}$ that is not a cycle
of $G_i$ must contain at least one
path in $G_i-\{v_1v_2,u_1u_2\}$ that connects
two distinct vertices among these four, its length is at least the
minimum between $1+(2r-1)$ and the girth of $G_i$ (in fact, twice
this girth plus $2$, but this is not crucial here). Therefore
$G_{i+1}$ also satisfies (3). By Corollary \ref{c23} 
condition (4) also holds for $i$.
Since each graph $G_{i+1}$ in the process has less crossing edges
than the previous graph  $G_i$ the process must terminate with a graph
$G_t$ in which the number of crossing edges is at most $\sqrt n$.
Let $\vv'$ be the vector assigning $1/\sqrt n$ to each vertex of
$B$ and $-1/\sqrt n$ to  each vertex of $C$. Then
$\vv'^T A_{G_t} \vv' \geq (dn-4\sqrt n)/n=d-4/\sqrt n$. Since
$\|\vv'\|_2^2=1$ and its sum of coordinates in $0$, this implies
that $\lambda_2(G_t) \geq d-4/\sqrt n$, showing that condition
(2) holds. This completes the proof.
\end{proof}

\subsection{The smallest eigenvalue}

%\noindent
\begin{proof}[Proof of Theorem \ref{t12}]
%{\bf Proof of Theorem \ref{t12}:}\, 
The proof is very similar to that of Theorem \ref{t11}, we thus
only include a brief description.  The fact that $A_s(d)$ is contained
in the closed interval $[-d, -2\sqrt{d-1}]$ follows from the result of
Li \cite{Li} (see also \cite{Ni1}) that asserts that the smallest 
eigenvalue of any $d$-regular graph in which the length of the 
shortest odd cycle is $r$, is at most $-2\sqrt{d-1}(1-O(1/r^2)).$
In order to show that every point of this interval indeed lies
in $A_s(d)$ we construct, for every even integer $n$, a sequence
$G_0,G_1, \ldots ,G_t$ of graphs on a set $V$ of $n$ vertices
that satisfy the following, where $c_1,c_2,c_3,c_4$ are positive
constants depending only on $d$.
\begin{enumerate}
\item
$\lambda_n(G_0) \geq -2\sqrt{d-1} -c_1(\frac{\log \log n}{\log n})^2.$
\item
$\lambda_n(G_t) \leq -d+\frac{c_2}{\sqrt n}.$
\item
The girth of each $G_i$ is at least $c_3 \log \log n$.
\item
For every $0 \leq i <t$, if the absolute value of 
either $\lambda_n(G_i)$ or $\lambda_n(G_{i+1})$
exceeds $2 \sqrt{d-1}$ then\,
$|\lambda_n(G_{i+1})-\lambda_n(G_i)| 
\leq \frac{c_4}{\log \log n}$.
\end{enumerate}
This clearly suffices to complete the proof of the theorem.
As in the previous proof, the existence
of $G_0$ satisfying properties (1) and (3) follows from the results of
\cite{Fr} and \cite{MWW}.  Assuming we have already constructed
$G_0, G_1, \ldots G_i$ satisfying (3) and (4), we obtain
$G_{i+1}$ 
from $G_i$ by a single swap. Split the set of vertices
of $G_0$ into two sets $B,C$, each of size $n/2$. As long as the number
of edges 
in the induced subgraph of $G_i$ 
on $B$ is at least, say, $\sqrt n$, so is the
the number of edges in the induced subgraph of $G_i$ on $C$.  Indeed,
these two numbers are equal since the sum of degrees in the induced
subgraph on $B$ is $|B|d-e(B,C)$ where $e(B,C)$ is the number of edges
connecting $B$ and $C$, and the sum of degrees in the induced subgraph on
$C$ is $|C|d-e(B,C)$.
As in the
previous proof, there are two edges $v_1v_2$ and $u_1u_2$ of $G_i$, where 
$v_1,v_2 \in B$, $u_1,u_2 \in C$ and the distance in $G_i$ between these
two edges is at least $\Omega(\log n/ \log d)$. Swapping these edges and 
replacing them by the two crossing edges $v_1u_1$ and $v_2u_2$ we obtain
a graph $G_{i+1}$ with girth satisfying condition (3). As the
assertion of Corollary
\ref{c23} clearly holds for the smallest eigenvalues too,
condition (4) also holds for  $i$.
Since each graph $G_{i+1}$ in the process has 2 more crossing edges
than the previous graph  $G_i$, the process must terminate with a graph
$G_t$ in which the number of non-crossing edges is at most $2\sqrt n$.
As in the previous proof, let 
$\vv'$ be the vector assigning $1/\sqrt n$ to each vertex of
$B$ and $-1/\sqrt n$ to  each vertex of $C$. Then
$\vv'^T A_{G_t} \vv'\leq (-dn+8\sqrt n)/n=-d+8/\sqrt n$. Since
$\|\vv'\|_2^2=1$, this implies
that $\lambda_n(G_t) \leq -d+8/\sqrt n$, showing that condition
(2) holds. This completes the proof of the theorem.%\hfill $\Box$
\end{proof}

\section{The top eigenvalues of bounded degree graphs}

\subsection{Bounded degree graphs}

In this subsection we prove Theorem \ref{t14}. To do so we establish
the following lemma.
\begin{lemma}
\label{l24}
For every $d \geq 2$ and every even integer $n$, and for every
real $\lambda \in [2\sqrt{d-1},d]$ there is a graph $G=G(n,\lambda)$ 
with maximum degree at most $d$, whose number of vertices is between
$\sqrt n$ and $n$, satisfying 
\begin{enumerate}
\item
$|\lambda_1(G) -\lambda| \leq 2 \frac{d\log n}{\sqrt n}$
\item
$\lambda_2(G) \leq 2\sqrt{d-1}$.
\end{enumerate}
\end{lemma}
\begin{proof}
To simplify the presentation we omit all floor and ceiling signs
whenever these are not crucial.
We show that there is a sequence of graphs
$G_0,G_1, \ldots ,G_t$, where $t=n-\sqrt n$, $G_i$ has exactly
$n-i$ vertices, so that
\begin{enumerate}
\item
$\lambda_1(G_0) =d$
\item
$\lambda_1(G_t) \leq 2\sqrt{d-1}+\frac{d}{\sqrt n}$
\item
For every $0 \leq i \leq t$, $\lambda_2(G_i) \leq 2 \sqrt{d-1}$.
\item
For every $0\leq i<t$:
$$
\lambda_{1}(G_{i})\left(1-3\frac{\log (n-i)}{n-i}\right) \leq \lambda_1(G_{i+1})
\leq \lambda_1(G_i).
$$
\end{enumerate}
This clearly implies the assertion of the lemma. A $d$-regular graph 
$G_0$ on
$n$ vertices satisfying (1) and (3) exists by the result of
Marcus, Spielman and Srivastava in \cite{MSS1}. Assuming we have already
defined $G_0,G_1, \ldots ,G_i$ satisfying (3) and (4), where
$G_j$  has $n-j$ vertices for all $j \leq i$ and where
$i+1<t$, we define $G_{i+1}$ as follows. Let $\ell$ be the largest
even integer that does not exceed $(n-i)/2$ and consider 
all closed walks of
length $\ell$ in $G_i$. By averaging, there is at least one vertex of 
$G_i$ contained in at most half of these walks. Choose arbitrarily such a 
vertex $v$ and let $G_{i+1}$ be the graph obtained from $G_i$ by removing
the vertex $v$. Thus $G_{i+1}$ has $n-i-1$ vertices. As it is an induced
subgraph of $G_i$ it satisfies condition (3) by eigenvalue interlacing
(see, e.g., \cite{Fi}). 
The same eigenvalue interlacing implies that $\lambda_1(G_{i+1})
\leq \lambda_1(G_i)$. In order to prove the other inequality
in condition (4) we use the following simple fact.
\vspace{0.1cm}

\noindent
{\bf Fact:}\, Let $H$ be a graph with $q$ vertices, and let $\ell$ 
any even positive integer. Let
$T=T(\ell)$ be the number of 
closed walks of length $\ell$ in $H$. Then
$\sum_{i=1}^q \lambda_i^{\ell}(H)=T$ and hence
$(T/q)^{1/\ell} \leq \lambda_1(H) \leq T^{1/\ell}$. 
\vspace{0.1cm}

\noindent
By the above fact applied to $G_i$, with $\ell$ being the largest
even integer that does not exceed $(n-i)/2$ and $T$ being the number
of closed walks of length $\ell$ in $G_i$, it follows that
$\lambda_1(G_i) \leq T^{1/\ell}$. Applying the fact to 
$G_{i+1}$ with the same $\ell$, and using the fact that the number
of closed walks of length $\ell$ in it is at least $T/2$, we conclude
that 
$$\lambda_1(G_{i+1}) \geq \left(\frac{T}{2(n-i-1)}\right)^{1/\ell} 
\geq T^{1/\ell} \cdot \left(1-\frac{ 3 \log (n-i)}{n-i}\right)
\geq 
\lambda_1(G_i) \cdot \left(1-\frac{ 3 \log (n-i)}{n-i}\right)
$$
where here we used the fact that $n-i$ is large and that 
$\ell$ is close to half of it. This shows that condition
(4) is maintained with $G_{i+1}$. 

It remains to prove that 
condition (2) holds. This 
follows from the argument in the first part of the proof of Lemma 9.2.7
in \cite{AS}. For completeness we sketch the proof.
Let $f$ be an eigenvector
corresponding to the largest eigenvalue of $G_t$. Let
$g$ be a vector defined on the vertex set of $G_0$, by
letting $g(v)=f(v)$ for every vertex $v \in V(G_t)$
and by defining $g(u)=0$ for all other vertices of $G_0$.
Expressing this vector $g$ as a linear combination of the 
all $1$-vector (which is the top eigenvector of 
$G_0$) and an orthogonal vector $h$, and estimating
$\lambda_1(G_t) = g^T A_{G_0} g$ using this expression and 
noting that $g^T 1 %\langle g, 1\rangle 
\leq \|g\|_2 n^{1/4} $ and $ h^T A_{G_0} h \leq \lambda_2(G_0) \|h\|_2^2$, we get the 
required estimate in condition (2). This completes 
the proof of the lemma.  %\hfill $\Box$
\end{proof}

\vspace{0.2cm}

\begin{proof}[Proof of Theorem \ref{t14}]
Let $(\mu_1,\mu_2, \ldots ,\mu_k)$ be a vector satisfying
$d \geq \mu_1 \geq \mu_2 \ldots  \geq \mu_k \geq 2 \sqrt{d-1}$.
By Lemma \ref{l24}, for every $i$, $1 \leq i \leq k$ and for every
even integer $n$ there is a graph $G_j=G(n,\mu_j)$ 
with maximum degree at most $d$, whose number of vertices is between
$\sqrt n$ and $n$, satisfying 
$|\lambda_1(G_j) -\mu_j| \leq 2 \frac{d\log n}{\sqrt n}$
and 
$\lambda_2(G_j) \leq 2\sqrt{d-1}$.
Let $G(n)$ be the vertex 
disjoint union of the graphs $G_1,G_2, \ldots ,G_k$.
Then $|\lambda_i(G(n)) -\mu_i| \leq \frac{2d \log n}{\sqrt n}$
for all $1 \leq i \leq k $. Any sequence of such graphs $G(n)$
for a growing sequence of values of $n$ gives the required limit point
$(\mu_1,\mu_2, \ldots ,\mu_k)$. Note that if $\mu_1=d$ then the 
maximum degree of each graph $G(n)$ is exactly $d$.  %\hfill $\Box$
\end{proof}
\medskip

\noindent
The graphs constructed in the proof of Theorem \ref{t14} are
not connected. We can show, however, that the same set $B(k,d)$ of vectors
is obtained by similar limits of the corresponding vectors of 
top eigenvalues of {\em connected} graphs with maximum degree at most $d$.
This requires some additional ideas. The details follow.
\begin{theo}
\label{t25}
For every $d \geq 2$ and every $k$, every point of $B(d,k)$ is a limit
point of a sequence of vectors 
$(\lambda_1(G_i),\lambda_2(G_i), \ldots ,\lambda_k(G_i))$
for an infinite sequence  $G_i$ of {\em connected} graphs $G_i$ with maximum
degree at most $d$.
\end{theo}
\vspace{0.2cm}

\noindent
To establish this theorem we
first prove the following variant of Lemma \ref{l24}.
\begin{lemma}
\label{l26}
There are positive constants $c_1=c_1(d),c_2=c_2(d),c_3=c_3(d)$ so that the
following holds.
For every $d \geq 2$ and every even integer $n$, and for every
real $\lambda \in [2\sqrt{d-1},d]$ there is a
graph $G=G(n,\lambda)$  satisfying the following.
\begin{enumerate}
\item
$G$ is connected, it 
has at least $n/\log n$ and at most $n-1$ vertices, its girth
is at least $c_1 \log \log n$, its maximum degree is at most $d$
and it has at least 
$2$ vertices of degree strictly smaller than $d$.
\item
$|\lambda_1(G) -\lambda| \leq \frac{c_2}{\log n}$
\item
$\lambda_2(G) \leq 2\sqrt{d-1}+c_3 \left(\frac{\log \log n}{\log n}\right)^2$.
\end{enumerate}
\end{lemma}
%\vspace{0.1cm}

%\noindent
%{\bf Proof:}\, 
\begin{proof}
As in the proof of Lemma \ref{l24} we construct 
a sequence of graphs
$G_0,G_1, \ldots ,G_t$, where $t=n-n/\log n$, and $G_i$ has 
maximum degree at most $d$ and exactly
$n-i$ vertices, so that
\begin{enumerate}
\item
$\lambda_1(G_0) =d$ 
\item
$\lambda_1(G_t) \leq 2\sqrt{d-1}+\frac{c_2}{\log n}$
\item
For every $0 \leq i \leq t$, 
$\lambda_2(G_i) \leq 2 \sqrt{d-1}+c_3 (\frac{\log \log n}{\log n})^2$.
\item
For every $0\leq i<t$:
$$
\lambda_{1}(G_{i})\left(1-3\frac{\log^5 n}{n}\right) \leq \lambda_1(G_{i+1})
\leq \lambda_1(G_i).
$$
\item
Each $G_i$ is connected, and has girth at least
$c_1 \log \log n$.  For $i\geq 1$ each $G_i$ has at least two
vertices of degree strictly  smaller than $d$.
\end{enumerate}
This easily implies the assertion of the lemma. A $d$-regular graph 
$G_0$ on
$n$ vertices satisfying the requirements in
conditions (1), (3) and (5) exists by the work of Friedman \cite{Fr}
and the results in \cite{MWW}, as explained in the beginning of the
proof of Theorem \ref{t21}. All the other graphs $G_i$ will
be induced subgraphs of $G_0$, where each $G_{i+1}$ is
obtained from $G_i$ by deleting a carefully chosen vertex.
Note that trivially all graphs $G_i$ will satisfy the girth
condition in (5). Moreover, as the initial graph $G_0$ is a 
$d$-regular strong
expander, it contains no cutpoints, and hence any nontrivial
connected induced subgraph of it contains at least $2$ vertices of degree
strictly smaller than $d$.
Since $G_0$ is a strong expander, its
diameter is at most $D = O(\log n)$. Any BFS tree in it starting from
an arbitrarily chosen root has at most $D+1$ levels. Fix 
such a tree $T_0$ in $G_0$. Assuming $G_j$ and a spanning tree
of it $T_j$ of diameter at most $D+1$ have been defined already
for all $j \leq i$, define $G_{i+1}$ as follows. Let
$\ell$ be the largest  even number which does not exceed,
say, $n/\log^4 n$. 
Let $v$ be a leaf of $T_i$ contained in the smallest
number of closed walks of length $\ell$ in $G_i$. Define
$G_{i+1}=G_i-\{v\}$, and $T_{i+1}=T_i-\{v\}$.  Note that 
any spanning tree in a graph of $m$ vertices which has at most
$D+1$ levels has at least $m/(D+1)$ leaves, since all its vertices can be
covered by all the root to leaf paths, and each
such path contains at most $D+1$ vertices. Since $\ell =o( m/((D+1)\log n ))$,  by averaging
in our process there would always be a leaf contained in at most
a fraction of $O(1/\log n)$ of the closed walk. The estimate
required in condition (4) thus follows, as in the proof of
Lemma \ref{l24}. Condition (2) also follows by the argument 
for establishing condition (2) in
the proof of Lemma \ref{l24}.  This completes the proof. %\hfill $\Box$
\end{proof}

\vspace{0.2cm}

\noindent
We now prove the following useful lemma, whose special case is 
Lemma 3.2 in \cite{Al1} 
(See also \cite{AGS}, \cite{Ka} for related arguments.)
\begin{lemma}
\label{lem:treemiddlelayer}
Let $H$ be a graph with maximum degree $d \geq 2$ and let $U$ be an
independent set of vertices. Suppose, further, that the induced 
subgraph of $H$ on the 
union of $U$ with the $(l+1)$ neighborhood of $U$ in
some connected component of $H \setminus U$ 
is a collection of $|U|$ vertex
disjoint trees, where the roots are the vertices of $U$. 
For each $i$ satisfying $0\leq i \leq l+ 1$, let $X_i$ be the set of
vertices of distance exactly $i$ from $U$ in those trees. Let  $\vv$ be a
nonzero eigenvector of $H$ with eigenvalue at least $2\sqrt{d-1}$. Then
for every $1 \leq i \leq l-1$ and any $1 \leq j \leq \min(l-i, i)$, 
\begin{equation}
       \sum_{u \in X_{i-j+1}} \vv(u)^2 + \sum_{u \in X_{i+j}} \vv(u)^2 \leq \sum_{u \in X_{i-j}} \vv(u)^2 + \sum_{u \in X_{i+j+1}} \vv(u)^2. \label{eq:sum2}
\end{equation}
As a consequence, for any $1 \leq i \leq l-1$,
\begin{equation}
    \sum_{u \in X_i \cup X_{i+1}} \vv(u)^2 \leq \frac{\|\vv\|_2^2}{\min(i+1, l-i+1)}. \label{eq:small}
\end{equation}
\end{lemma}
\begin{proof}
Let $u \in X_i$ where $1\leq i \leq l$, and let 
$u'$ be its unique neighbor in $X_{i-1}$. Then  %Since $\vv$ is an eigenvector,
\begin{equation}
    \lambda \vv(u) = \sum_{w \in X_{i+1}, w \sim u} \vv(w) + \vv(u'). 
\end{equation}
 Thus  by writing $\vv(u') = (d(u')-1)\vv(u')/(d(u')-1)$ and by Cauchy-Schwarz,  we have 
\begin{equation}   
    \sum_{w \in X_{i+1}, w \sim u} \vv(w)^2 + \frac{\vv(u')^2}{d(u')-1} \geq  \frac{\lambda^2 \vv(u)^2}{ d(u)+d(u')-2}   \geq 2\vv(u)^2. 
\end{equation} 
(Here it is convenient to denote, for $u' \in U$, by $d(u')-1$ the 
degree of $u'$ as a root of the corresponding tree described 
in the lemma, even when the actual degree of $u'$ in  $H$ may be
larger than $d(u')$. This is convenient for the uniformity of the
notation, and the only property needed in the proof is that
$d(u') \leq d$ with this notation too, which clearly holds).

Add up these inequalities for all $u \in X_i$. Noticing that each vertex $w$ in $X_{i+1}$ is adjacent to exactly one vertex in $X_i$ while each vertex $u'$ in $X_{i-1}$ is adjacent to exactly $(d(u')-1)$ vertices in $X_i$, 
putting $S_i =\sum_{u \in X_i} \vv(u)^2$, it follows  that
\begin{equation}
    2\sum\nolimits_{u \in X_i} \vv(u)^2 \leq \sum\nolimits_{w \in X_{i+1}} \vv(w)^2 + \sum\nolimits_{u' \in X_{i-1}} \vv(u')^2 \implies 2S_i \leq S_{i+1} + S_{i-1}.  \label{eq:2si}
\end{equation}
Thus if $1 \leq i \leq l-1$, by adding (\ref{eq:2si}) for $i$ and for $i+1$, 
we have 
\begin{equation}
    2S_i + 2S_{i+1} \leq S_{i+1} + S_{i-1} + S_{i+2}+ S_i  \implies S_i + S_{i+1} \leq S_{i-1} + S_{i+2}.\label{eq:sumSi}
\end{equation}

We now prove (\ref{eq:sum2}) by induction on $j$. 
The base case when $j =1$ is (\ref{eq:sumSi}). Suppose the claim holds up to $j-1$ where $j \geq 2$. 
Apply (\ref{eq:2si}) to $S_{i-j+1}$ and $S_{i+j}$ to get that
$
 2(S_{i-j+1} + S_{i+j}) \leq   S_{i-j}+S_{i-j+2}  + S_{i+j-1}+S_{i+j+1}
 \leq S_{i-j+1} + S_{i+j}+ S_{i-j}+S_{i+j+1}
$  by the inductive hypothesis. Therefore $S_{i-j+1} + S_{i+j} \leq S_{i-j}+S_{i+j+1} $, as desired. 

Finally, (\ref{eq:small}) is proved by applying (\ref{eq:sum2}) with $j = 1, \dots, \min(l-i, i)$. 
\end{proof}
\begin{lemma}
\label{c28}
Let $F$ and $H$ be two connected graphs,
each having girth at least $2r+1$
and maximum degree at most $d$. Let $G$ be the graph obtained 
from the vertex disjoint union of $F$ and $H$ by adding 
an arbitrary edge connecting them, keeping the maximum degree
at most $d$.
Let $\mu_1 \geq \mu_2 \geq \ldots \geq \mu_s$ be the $s$ largest
eigenvalues of the graph which is the disjoint union of
$F$ and $H$ (that is, the $s$ largest elements in the set
of all eigenvalues of $F$ and all eigenvalues of $H$,
taken with multiplicities). Then 
$$
|\mu_s - \lambda_s(G)| \leq \frac{2 s}{r+1}.
$$
\end{lemma}
%\vspace{0.2cm}

\begin{proof}
For each eigenvalue $\mu_i$ above which is an eigenvalue of $H$, 
let  $f_i$ be the corresponding normalized
eigenvector viewed as a vector defined on $V(G)=V(H) \cup  V(F)$,
by extending its definition to be $0$ on $V(F)$. Similarly, if
$\mu_i$ is an eigenvalue of $F$ let $f_i$ be a corresponding eigenvector
defined to be $0$ on the vertices of $H$. These vectors 
are the normalized top $s$ eigenvectors of $H \cup F$
and span a subspace of dimension $s$. Let $A'=A_{H \cup  F}$ be 
the adjacency
matrix of the disjoint union of $H$ and $F$, then for any normalized
vector $\yy$ in this subspace $\yy^T A' \yy \geq \mu_s$.
By Claim \ref{cl22}, the $\ell_{\infty}$ norm of each of the vectors
$f_i$ is at most $1/\sqrt{r+1}$ and hence by Cauchy-Schwarz
the $\ell_{\infty}$ norm of each normalized vector $y$ in this
space is at most
$\sqrt{s/(r+1)}$. Since the graph $G$ is obtained from
$H \cup F$ by the addition of a single edge, for each such
$\yy$, $\yy^T A_G \yy$ and $\yy^T A' \yy$ differ  by only two terms
of the form $\yy(u)\yy(v)$ and hence 
$\yy^T A_G \yy \geq \mu_s -2s/(r+1)$. By the variational definition
of $\lambda_s(G)$ this implies that $\lambda_s(G)
\geq \mu_s -2s/(r+1)$. To upper bound $\lambda_s(G)$ consider 
the subspace $W$ spanned by the eigenvectors of the top
$s$ eigenvalues of $G$. This subspace contains a nonzero
normalized vector $\zz$
orthogonal to all the vectors $f_i$ defined above 
for $1 \leq i \leq s-1$. In addition, its $\ell_{\infty}$-norm
is at most $\sqrt{s/(r+1)}$. It is clear that
$\zz^T A' \zz \leq \mu_s$ and as before the fact that
$\|\zz\|_{\infty} \leq \sqrt{s/(r+1)}$ implies that 
$\lambda_s(G) \leq 
\zz^T A_{G} \zz \leq \mu_s + 2s/(r+1)$. This supplies the desired 
upper bound for $\lambda_s(G)$, completing the proof. %\hfill $\Box$
\end{proof}
\vspace{0.2cm}

\begin{proof}[Proof of Theorem \ref{t25}]
Let $(\mu_1, \ldots \mu_k)$ be a vector in $B(d,k)$. By Lemma
\ref{l26} there are connected graphs $G_j=G(n,\mu_j)$ with the 
following properties. 
\begin{enumerate}
\item
Each $G_j$ has maximum degree at most
$d$ and has at least two vertices of degree less than $d$.
\item
$|\lambda_1(G_j)-\mu_j| \leq \frac{c_2}{\log n}$
\item
$\lambda_2(G_j) \leq 2\sqrt{d-1}+c_3 (\frac{\log \log n}{\log n})^2$.
\item
The girth of $G_j$ is at least $c_3 \log \log n$.
\end{enumerate}
Pick two vertices $u_{j,1}, u_{j,2}$ of degree less than $d$ in each $G_j$ 
for $2 \leq j \leq k-1$, one such vertex $u_{1,2}$ in 
$G_1$ and one such vertex $u_{k,1}$ in $G_k$. Adding the edges
$u_{j,2}u_{j+1,1}$ for $1 \leq j<k$ to 
the vertex disjoint union of the graphs $G_j$
we get a connected graph $G=G(n)$ with maximum degree at most $d$. 
Applying 
Lemma \ref{c28} $k-1$ times it follows that if $\lambda_1 \geq
\lambda_2 \geq \ldots \geq \lambda_k$ are its top $k$ eigenvalues then
$|\lambda_i -\mu_i| \leq \frac{16k^2}{c_3 \log \log n}$
for all $1 \leq i \leq k$. 
Taking a growing sequence of values of $n$ completes
the proof of the theorem. %  \hfill $\Box$
\end{proof}

\section{The top eigenvalues of regular graphs}
We start with some
notation. Let $H=(V,E)$ be a graph with maximum degree 
at most $d$. 
The {\it $d$-augmentation} $A_d(H)$ is the graph obtained from
$H$ as follows. For each vertex $v$ of degree $d(v)<d$ of
$H$, let $U_v$ be a set of $d-d(v)$ new vertices, where all
sets $U_v$ are pairwise disjoint. Every point in $U_v$ is adjacent
to the vertex $v$. Thus, if $H$ is $d$-regular then $A_d(H)=H$. In
every other case all vertices of $A_d(H)$ are of degrees $d$ or $1$,
and the number of leaves (vertices of degree $1$) is exactly
$d|V|-2|E|$.   Put $T^0(d) H =H$,  $T^1(d) H = A_d(H)$ 
and $T^{i+1}(d)H=A_d(T^i(d)H)$ for all $i \geq 1$. Note that $T^r(d)H$
is obtained from the vertex disjoint union of $H$ and 
$(d|V|-2|E|)$ trees with $r$ levels by joining the roots of these 
trees to the vertices of $H$ of degree lower than $d$. 

\subsection{The top eigenvalues of regular graphs -- a simple version}

In this subsection we present the proof of the following Theorem \ref{t15}, 
which is a less powerful version of Theorem \ref{thm:limit}, 
using a simpler analysis based on the same basic approach. A 
similar method with slightly more sophisticated arguments
will be used later to establish Theorem \ref{thm:limit}.

\begin{theo} \label{t15}
For every $d \geq 3$ and every $k$, every point of 
$$
C(d,k)=\{(\mu_1,\mu_2, \ldots ,\mu_k):~d=\mu_1 \geq 
\mu_2 \geq \mu_3 \geq \ldots \geq \mu_k \geq 3\sqrt{d-1} \}.
$$ 
is a limit
point of a sequence of vectors 
$(\lambda_1(G_i),\lambda_2(G_i), \ldots ,\lambda_k(G_i))$
for an infinite sequence  $G_i$ of $d$-regular graphs.
\end{theo}
We need the 
following lemma.
\begin{lemma}
\label{l29}
Let $d \geq 3$ be an integer
and let $\eps>0$ be a real number. Then for every real $\lambda$ satisfying
$2 \sqrt{d-1}+2\eps \leq \lambda \leq d$ there exists a graph $H$
with maximum degree at most $d$ in which every connected
component has at least one vertex of degree smaller than $d$
satisfying the following properties.
\begin{enumerate}
\item
The girth of $H$ is at least $20 d/\eps$.
\item
The number of leaves of $T^1(d)H$ is divisible
by $2d$.
\item
$\lambda_2(H) \leq 2 \sqrt{d-1}+\eps$.
\item
For every $i \geq 4 \sqrt{d-1}/\eps$, 
$|\lambda_1(T^i(d)H) -\lambda| \leq 2 \eps.$
\end{enumerate}
\end{lemma}
\begin{proof} As in the previous proofs we construct a family of 
graphs $H_0, H_1, \ldots $ and show that $H$ can be chosen to 
be one of them.  We start with a $d$-regular high-girth near Ramanujan
graph $G$ with a large even number $m$ of vertices, and omit from it 
two nonadjacent
vertices to get a graph $H_0$ satisfying (1) and (2). Each other graph
$H_i$ of the sequence will be obtained from the previous one by
deleting two nonadjacent vertices and by 
adding, if needed, a set of at most $d$ isolated edges
to ensure that the number of leaves of $T^1(d)H_i$ is divisible
by $2d$. Thus all these graphs satisfy (1) and (2).
To prove (3) note that $H_{i+1}$ is obtained from the vertex disjoint
union of an induced subgraph of $G$ and a collection of  isolated
edges and hence (3) follows by eigenvalue interlacing and the fact
that $G$ is near Ramanujan.
It remains to analyze the largest
eigenvalues of the graphs $H_i$ and their augmentations. All of these graphs
have largest degree at most $d$ and hence the top eigenvalue of all is
at most $d$. In addition, the top eigenvalue of $H_0$ is very close
to $d$ as its average degree is very close to $d$. Note also that
eigenvalue interlacing implies that $\lambda_1(T^j(d)H_i)$ is an 
increasing function of $j$ for every fixed $i$ showing that
$\lambda_1(T^j(d) H_0)$ is very close to $d$ for all $j$.  Consider
some fixed $j $. Note that $T^j(d) H_{i+1}$ is obtained
from $T^j(d) H_i$ by deleting two vertices, omitting several connected
components each of which is a $d$-tree, and adding at most $3d$ such
components, joining the root of some of them to the graph.
By Claim \ref{cl22} and the high girth
the omission of two vertices does not change $\lambda_1$ by much.
The subsequent removal and addition 
of the connected components which are trees does not 
change it at all, as long as the largest eigenvalue exceeds
$2\sqrt{d-1}$. The addition of the edges also hardly changes it,
by Lemma \ref{c28}. Thus the difference between 
$\lambda_1(T^j(d) H_i)$ and $\lambda_1(T^j(d) H_{i+1})$ is smaller
than $\eps$. Since the graphs $H_i$ end with one which is a 
union of disjoint trees
for which $\lambda_1(T^j(d) H_i) \leq 2\sqrt{d-1}$, for every 
fixed $j$ we can find some $i$ so that $\lambda_1(T^j(d) H_i)$
is within $\eps$ of $\lambda$. It remains to show the following claim.
\begin{claim}\label{claim:notgrow}
Let $J = \lceil 4 \sqrt{d-1}/\eps\rceil$. Assume the largest 
eigenvalue of $T^J(d)H_i$ is at least $2\sqrt{d-1}+\epsilon$. 
For any $R \geq J$, $|\lambda_1(T^J(d) H_i)- \lambda_1(T^R(d) H_i)| 
\leq \eps$,  namely,
$\lambda_1(T^j(d) H_i)$ hardly grows after $j=4 \sqrt{d-1}/\eps$.  
\end{claim}
\begin{proof}
By the monotonicity of $\lambda_1(T^j(d)H_i)$ as $j$ increases, it suffices to show 
that for arbitrarily
large $R$ there is some $j \leq 4 \sqrt{d-1}/\eps$ so that
$|\lambda_1(T^j(d) H_i)- \lambda_1(T^R(d) H_i)| \leq \eps$, namely,
that $\lambda_1(T^j(d) H_i)$ hardly grows after $j=4 \sqrt{d-1}/\eps$. 

Let $H'$ be a connected component of $T^R(d)H_i$ with the maximum
eigenvalue of it, call it $\mu$, and let $A'$ be its adjacency matrix. By assumption, $\mu \geq  2\sqrt{d-1}+\eps$. 
Let
$\vv$ be the normalized eigenvector of the maximum 
eigenvalue $\mu$ of $H'$. 

For each $i \leq R$ let $s_i$ denote the sum of squares of the entries
of $\vv$ on the vertices of $H'$ that are of distance exactly 
$i$ from  the set of 
vertices $U$ 
of the original graph $H_i$. (Call these vertices the vertices at level $i$).
Thus $\sum_i s_i=1$ and therefore there is
some index $1 \leq i \leq 4\sqrt{d-1}/\eps$ so that $s_i+s_{i+1} \leq
\eps/(\sqrt{d-1})$.  Let $\vv_s$ be the restriction 
of the vector $v$
to the vertices of distance at most $i$ from $U$, 
$\vv_m$ the restriction of this
vector to the two consecutive levels $i,i+1$, 
and $\vv_{l}$ it's restriction
to the levels at least $i+1$. 
With some abuse of notation consider each of these
three vectors as one defined on all vertices of $H'$, where the coordinates
are set to $0$ in the irrelevant levels. Note that since every level
$j \geq 1$ is an independent  set it follows that
$$
\mu=\vv^T A'\vv=\vv_s^T A' \vv_s+\vv_m^t A' \vv_m+ \vv_{l}^T A' \vv_{l}.
$$
However
$\vv_m^T A' \vv_m \leq \sqrt{d-1} \|\vv_m\|_2^2$ since the induced
subgraph on the two levels $i$ and $i+1$ is a union of vertex disjoint
stars, each with maximum degree at most $d-1$. Since $\|\vv_m\|_2^2 \leq
\eps/\sqrt{d-1}$ this is at most $\eps$. In addition
$\vv_{l}^T A' \vv_{l} \leq 2\sqrt{d-1} \|\vv_{l}\|_2^2$, as the induced
subgraph on the vertices at levels exceeding $i$ is a union of $d$-trees.
It follows that if $\mu \geq 2\sqrt{d-1}+\eps$ then
$\vv_s^T A' \vv_s \geq (\mu-\eps) \|\vv_{s}\|_2^2$ since otherwise
the sum of all three terms is smaller than
$$
(\mu-\eps)\|\vv_s\|_2^2+\eps + 2\sqrt{d-1} (1-\|\vv_s\|_2^2) 
\leq (\mu-\eps)+\eps,
$$ 
contradiction. 
\end{proof}
On the other hand, if $\lambda_1(T^J(d)H_i) < 2\sqrt{d-1}+\eps$, taking $H$ to be a collection of $d$
isolated edges satisfies all requirements.
This completes the proof of the lemma.  %\hfill $\Box$
\end{proof}

\subsubsection{Proof  of Theorem \ref{t15}} 
We first prove the following patching lemma. 

For each $1 \leq i \leq k-1$, let $F_i = (V_i, E_i)$ be a graph with maximum degree at most $d$. Let $F_0$ be a $d$-regular graph with girth at least $8R$. By an $R$-{\it patching} of $F_1, \dots, F_{k-1}$ to $F_0$, we mean a graph $G$ constructed in the following way. Suppose $|V(F_0)|$ is large enough to ensure that $F_0$ contains a collection $M$ of vertices such that $|M|d$ is the total number of leaves of all graphs $T^1(d)F_i$, $1 \leq i \leq k-1$, and the distance in $F_0$ between any two vertices in $M$ exceeds $4R$. Suppose the total number of such leaves is divisible by $2d$. The graph $G$ is obtained from the graphs $F_i =(V_i, E_i)$
as follows. Let $G_0$ be the induced subgraph of $F_0$
obtained by removing all vertices of $M$, where $M$ is as above.  This graph has exactly $|M|d$ vertices of degree $d-1$. On each set of vertices
$V_i$ for $1 \leq i \leq k-1$ we take a copy of $F_i$, and extend
it to a copy of $T^1(d)F_i$ by identifying the leaves of
$T^1(d)F_i$ with the required number of vertices of degree
$d-1$ of $F_0$. Clearly $G$ is a $d$-regular graph. 

\begin{lemma}\label{lem:patchingshort}
Let $d, R\geq 3$ be integers. For each $1 \leq i \leq k-1$, 
let $F_i'$ be a graph with maximum degree at most $d$ (which is
not $d$-regular) and girth at least $4R$, 
and let $F_i = (V_i, E_i)$ be $T^{8R}(d) F_i'$.
    Let $F_0 = (V_0, E_0)$ be a $d$-regular graph with girth at 
least $8R$.  For each $0\leq i \leq k-1$ let $\mu_i$ be the largest 
eigenvalue of $F_i$, and let $\mu$ be the maximum of the second 
largest eigenvalue of $F_i$ for $0 \leq i \leq k-1$. 
Suppose $\mu_1 \geq \mu_2 \dots \geq \mu_{k-1} 
\geq \max(2\sqrt{d-1},\mu) $. 
    Let the graph $G$ be an $R$-patching of $F_1, \dots, F_{k-1}$ 
to $F_0$ and let
    $\lambda_i$ be the $i$-th largest eigenvalue of $G$. Then for each $1 \leq i \leq k$, $|\lambda_i - \mu_{i-1}| \leq  \max(\sqrt{d-1}/R, 2d^3 \frac{\sum_{i=1}^{k-1}|V_i|}{|V_0|})$. 
\end{lemma}
\begin{proof}
We first prove the following easier direction. 
\begin{claim}
\label{c91}
The ordered $k$-tuple $(\lambda_1, \lambda_2,  \dots, \lambda_k)$ 
is at least the $k$-tuple 
$$(d-2d^3 \frac{\sum_{i=1}^{k-1}|V_i|}{|V_0|}, \mu_1, \dots, \mu_{k-1})$$
when arranged in descending order.
%For every $1 \leq i \leq k$, $\lambda_{i} \geq \mu_{i-1} - \sqrt{d-1}/R$.
\end{claim}
\begin{proof}
As $G$ and $F_0$ are $d$-regular, $\lambda_1 = \mu_0 =d$. 
Let $V_0'$ be the set of all vertices of $G$ of distance at least two from $\bigcup_{i=1}^{k-1} V_i$. Clearly $V_0' \subset V_0$. Since $F_0$ is sufficiently large, $|V_0 \setminus V_0'|$ is tiny and thus the average degree of the induced subgraph $F_0[V_0']$ exceeds $(d |V_0| - 2  \sum_{i=1}^{k-1}|V_i| d^3)/|V_0|$ . Therefore the maximum eigenvalue of $F_0[V_0']$ exceeds $d-  2  \sum_{i=1}^{k-1}|V_i| d^3/|V_0|$.  In addition, for every $1 \leq i \leq k-1$, the induced subgraph of $G$ on $V_i$ is $F_i$, with maximum eigenvalue $\mu_i$. In the induced subgraph of $G[V_0' \cup V_1 \dots \cup V_{k-1}]$ there are no edges between different sets and thus all the eigenvalues of the $k$ induced subgraphs of $G$ on $V_0', V_1, \dots, V_{k-1}$ are also eigenvalues of this induced subgraph. The assertion of the claim follows by eigenvalue interlacing.
\end{proof}

It remains to prove the upper bound.
\begin{claim}
\label{c92}
For every  $1 \leq i \leq k$, $\lambda_i \leq \mu_{i-1}+ \sqrt{d-1}/R$.
\end{claim}
\begin{proof}
Since  $\lambda_1 = \mu_0 =d$, fix $2 \leq  i \leq k$. If $\lambda_i \leq 2\sqrt{d-1}$, we are done by the assumption $\mu_{i-1} \geq 2\sqrt{d-1}$. Thus assume $\lambda_i > 2\sqrt{d-1}$.  
Let $V_0'$ be  $V(G) \setminus \bigcup_{j=1}^{k-1} V_j$. 
Let $G'$ be the subgraph of $G$ obtained by removing the edges between $V_0'$ and $\bigcup_{j=1}^{k-1} V_j$. Thus $G'$ is the disjoint union of $k$ graphs $F_0[V_0'], F_1, \dots, F_{k-1}$. 
Let $A, A'$ be the adjacency matrices of $G$ and $G'$ respectively. %Let $A_j$ be the adjacency matrix of $F_j$. 
Let $X$ be the set of vertices in $G$ incident to the edges across $V_0'$ and $\bigcup_{i=1}^{k-1}V_j$. Therefore $G$ restricted to $X$ is a union of vertex disjoint stars with maximum degree at most $d$. Let $C$ be the matrix obtained from $A$ by replacing all the entries of the rows and columns corresponding to vertices outside $X$ by $0$. Thus $A = A'+C$. 

Let $\sigma_1 \geq \sigma_2 \geq \dots $ be the eigenvalues of $G'$. Thus $\sigma_1 \leq d$.  As $\mu_1 \geq \mu_2 \dots \geq \mu_{k-1} \geq \max(2\sqrt{d-1},\mu)$, for each $2 \leq j\leq k$, $\sigma_j \leq \mu_{j-1}$. We will use $\sigma_i$ to upper bound $\lambda_i$.

For each $1 \leq j \leq k$, let $\vv_j$ be the normal eigenvector of $A$ corresponding to $\lambda_j$. Thus $\sqrt{N} \vv_1= \textbf{1}$ where $N = |V(G)|$. 
 By Lemma \ref{lem:treemiddlelayer} where the set $U$ is 
the set of vertices in $\bigcup_{i=1}^{k-1}V_i$ 
of distance exactly $R$ from $X$,  it follows that 
$
     \sum_{u \in X} \vv_j(u)^2 \leq 1/R. 
 $
As a consequence, as $C$ corresponds to a  disjoint union of stars of maximum degree at most $d$, 
 \begin{equation}
     |\vv_j^T C \vv_j| \leq \sqrt{d-1}  \sum\nolimits_{u \in X} \vv_j(u)^2 \leq \sqrt{d-1}/R. \label{eq:stars}
 \end{equation}
 
 Let $W$ be the $i$-dimensional space spanned by $\vv_1, \dots, \vv_{i}$. By the min-max principal, 
\begin{equation}
     \sigma_i \geq  \min\{ \vv^T A' \vv, \ \text{where } \vv \in W, \ 
  \|\vv\|_2^2=1\}. 
\end{equation}
  Write $\vv = \sum\nolimits_{j=1}^{i} c_j \vv_j$, where $\sum\nolimits_{j=1}^{i}c_j^2 = 1$. Then \[\vv^T A' \vv = \sum_{j=1}^{i}c_j^2 \vv_j^T A' \vv_j = \sum_{j=1}^{i}c_j^2 \vv_j^T A \vv_j - \sum_{j=1}^{i}c_j^2 \vv_j^T C\vv_j \geq \sum_{j=1}^{i}c_j^2 \vv_j^T A \vv_j - \sqrt{d-1}/R\] where the last inequality is by (\ref{eq:stars}). 
Since $\vv_j^T A \vv_j = \lambda_j$, we thus have 
\[
\sigma_i\geq \min\{ \sum\nolimits_{j=1}^i c_j^2 \lambda_j-\sqrt{d-1}/R, \ \text{where } \sum\nolimits_{j=1}^{i} c_j^2 = 1\} = \lambda_i - \sqrt{d-1}/R.  
\]
Together with $\sigma_i \leq \mu_{i-1}$ showed earlier, we have shown $\lambda_i \leq \mu_{i-1} + \sqrt{d-1}/R$. 
\end{proof}
\end{proof}

We can now prove Theorem \ref{t15}. 
\vspace{0.2cm}

\noindent
\begin{proof}[Proof of Theorem \ref{t15}]
For each sufficiently small $\eps$, it suffices to prove the result for all $\mu_1 > \mu_2 > \dots > \mu_k > 3\sqrt{d-1} + \eps$. 
Let $R = \lceil 4kd/\eps \rceil$. For each $1 \leq i \leq k-1$, let $F_i'$ be a graph satisfying the assertions of Lemma \ref{l29} for $\lambda = \mu_i$ and let $F_i = T^{8R}(d) F_i'$. Let $F_0$ be a sufficiently large near Ramanujan $d$-regular graph of girth at least $8R$. Theorem \ref{t15} will be proved by applying Lemma \ref{lem:patchingshort} to the $R$-patching of $F_1, \dots, F_{k-1}$ to $F_0$. To check all the assumptions in Lemma \ref{lem:patchingshort} are satisfied, 
it is only needed to show that the second largest eigenvalues of $F_0, F_1, \dots, F_{k-1}$ are at most $3\sqrt{d-1} + \eps$. Since $F_0$ is near Ramanujan, $\lambda_2(F_0) \leq 2\sqrt{d-1} + \eps < 3\sqrt{d-1}$.   Fix $1 \leq i \leq k-1$. In $F_i$, let $U_i$ be the vertex subset  $V(F_i) \setminus V(F_i')$. The induced subgraph of $F_i$ on $U_i$ is a disjoint union of trees of maximum degree at most $d$. Let $B$ be the adjacency matrix of 
the disjoint union of the graph $F_i'$ and $F_i[U_i]$. Thus $\lambda_2(B)\leq 2\sqrt{d-1} + \eps$. Let $C$ be the adjacency matrix of $F_i$ 
in which the only nonzero entries correspond to the cross-edges 
between $U_i$ and $V(F_i')$.
The matrix $C$ corresponds to a disjoint 
union of stars of degree at most $d$. 
Thus $\lambda_1(C) \leq \sqrt{d-1}$. Since the adjacency matrix of $F_i$ is $B + C$ and  $\lambda_2(B+C) \leq \lambda_2(B) + \lambda_1(C)$, the second largest eigenvalue of $F_i$ is at most $3\sqrt{d-1} + \eps$, as desired. 
\end{proof}

\subsection{Proof of Theorem \ref{thm:limit}}
Drawing inspiration from the proof of Theorem \ref{t15} as an application of Lemma \ref{lem:patchingshort}, we seek to refine our estimation of the second largest eigenvalue of $T^\ell(F_i)$ for sufficiently large $\ell$, given that $\lambda_2(F_i)$ is small. In the proof of Theorem \ref{t15}, we present a short argument to establish that $\lambda_2(T^\ell(F_i)) \leq 3\sqrt{d-1}+\eps$. The main ingredient in the proof of Theorem \ref{thm:limit} is the following lemma, which provides an almost optimal upper bound for the second largest eigenvalue $\lambda_2(T^\ell(d)F_i)$.

\begin{lemma}\label{lem:maingadget}
Fix an integer $d \geq 3$. Let $\epsilon>0$, and $R \geq 100d/\epsilon$. 
For any  $z \in (2\sqrt{d-1}, d)$, there is a graph $\tilde G$ with maximum degree at most $d$ and girth at least $R$ satisfying the following properties simultaneously.
\begin{enumerate}
    \item The largest eigenvalue of $T^\ell(d) \tilde G$ is within a range of $\epsilon$ to $z$ for any $\ell \geq \lceil 10\sqrt{d-1}/\epsilon \rceil$. 
    \item The second largest eigenvalue of $T^\ell(d) \tilde G$ is at most $2\sqrt{d-1} + \frac{1}{\sqrt{d-1}} +\epsilon/2$ for any $\ell \geq 0$.  
\end{enumerate} We may also assume the number of vertices of degree one in $T^1(d) \tilde G$ is divisible by $d$. 
\end{lemma}

Assuming Lemma \ref{lem:maingadget} holds, Theorem \ref{thm:limit} 
can be proved by using Lemma \ref{lem:patchingshort}. 
\vspace{0.2cm}

\noindent
\begin{proof}[Proof of Theorem \ref{thm:limit}]
 Apply the patching lemma (Lemma \ref{lem:patchingshort}) to assemble the augmentations of graphs that are guaranteed by Lemma \ref{lem:maingadget} and a large near Ramanujan graph. Theorem \ref{thm:limit} follows in a nearly identical way to the proof of Theorem \ref{t15}, which also utilizes Lemma \ref{lem:patchingshort}. Consequently, the details are omitted.
\end{proof}

\subsubsection{Proof of Lemma \ref{lem:maingadget}}
 To simplify the analysis, we make the following definition similar to $T^\ell(d)F$. 
 For any non-negative integers $s, \ell$ and a graph $F$, define the {\it $(d,s,\ell)$-augmentation} $T_s^\ell(d) F$ to be the graph obtained from the vertex disjoint union of $F$ and $s|V(F)|$ number of $d$-ary trees 
 with $\ell$ levels 
 \footnote{A {\it $d$-ary trees} is 
 a tree $T$ where the root has degree $d-1$ and each vertex except the root and leaves has $d-1$ children.  A $d$-ary tree is said to have $\ell$ {\it levels} if the distance between each leaf and the root is $\ell-1$. }
 by joining each vertex of $F$ to $s$ of these trees by edges such that each tree joins exactly one vertex. %the roots of these trees as an induced subgraph and where each vertex in $G'$ grows out $s$ number of $d$-ary trees with $\ell$ levels. (This means add $s$ disjoint $d$-ary trees and join each vertex in $G'$ to the roots of these $s$ trees. Thus the distance between the leaves and the vertex is $\ell$). 
Note that if $F$ is $d'$-regular with $d' \leq d$, then $T^\ell(d)(F) = T_{d-d'}^\ell(d)F$. Our basic gadgets will be of the form $T_{s}^\ell(d) G'$ where $G'$ is a regular graph. By interpolating between two basic gadgets, we will prove the existance of $\tilde G$ as desired in Lemma \ref{lem:maingadget}. We also sometimes write $T_s^\ell$ instead of $T_s^\ell(d)$ when there is no confusion. 
 
 The next lemma relates explicitly the top eigenvalues of a graph $G'$ and $T_s^\ell(d)G'$. For each non-negative integer $i$, define an auxillary function $a_i=a_i(\lambda)$ as 
 \begin{align}
a_i = 
\frac{1}{\sqrt{\lambda^2 - 4(d-1)}}\left(\left(\frac{\lambda+\sqrt{\lambda^2-4(d-1)}}{2}\right)^{i+1}
 -
\left(\frac{\lambda-\sqrt{\lambda^2-4(d-1)}}{2}\right)^{i+1}\right).  \label{eq:ai}
\end{align}
 
\begin{lemma}\label{lem:gap}
Let $d\geq 3, \ell \geq 2, s \geq 0$ be  integers, and let $G'$ be an arbitrary connected graph with the largest eigenvalue being $\mu_1$. Then the following equation has at most one 
solution $\lambda$ that is larger than $2\sqrt{d-1}$:
\begin{equation}
    \lambda - sa_{
    \ell-1}(\lambda)/a_\ell(\lambda) = \mu_1.  \label{eq:d'}
\end{equation}   
Furthermore, the top two eigenvalues of $T^\ell_s(d) G'$ satisfy the following statements.
\begin{enumerate}
    \item If (\ref{eq:d'}) has a solution larger than $2\sqrt{d-1}$, then this solution is the largest eigenvalue of $T^\ell_s(d) G'$.
    \item If (\ref{eq:d'}) has no solution larger than $2\sqrt{d-1}$, then $T^\ell_s(d) G'$ has no eigenvalue larger than $2\sqrt{d-1}$.
    \item For any $\epsilon \geq 0$, if the second largest eigenvalue of $G'$ is  strictly less than
%$2\sqrt{(1+\epsilon)(d'-1)}$ 
$2\sqrt{d-1+\epsilon} - \frac{s}{\sqrt{d-1+\epsilon} + \sqrt{\epsilon}},$
then the second largest eigenvalue of $T^\ell_s(d) G'$ is less than  $2\sqrt{d-1 + \epsilon}$.
\end{enumerate}
\end{lemma}
To prove Lemma \ref{lem:gap}, we need to establish two claims. The first claim demonstrates that in the graph $T^\ell_s(d)(G')$, an explicit relationship exists among the entries of an eigenvector on each of the $d$-ary trees. We refer to a vector that is supported on the vertices of a tree $T$ as {\it radial} if the vertices at the same distance from the root in $T$ have identical values in the vector.

\begin{claim}\label{claim:treeentry}
Let $d \geq 3$ be a fixed integer. Let $G = (V, E)$ be a graph and 
let $v \in V$ be a fixed vertex. Suppose some component in 
$G \setminus \{v\}$ is a $d$-ary tree $T$ of $\ell+1$ levels 
where $\ell \geq 0$ and in $T$ the root is the only neighbor of $v$. 
If $\lambda >2\sqrt{d-1}$ is an eigenvalue of $G$  then the eigenvector $\vv$ of $\lambda$ is radial on $T$. Furthermore, if the entries of $\vv$ corresponding to the leaves of $T$ have the value $x$, then for each $0 \leq i \leq \ell$, the entries of $\vv$ corresponding to vertices in $T$ that are at a distance of $\ell-i$ from the root have the value $a_{i} x$, where $a_i = a_i(\lambda)$ is defined as in (\ref{eq:ai}). 
\end{claim}
\begin{proof}
We prove the claim by induction on $\ell$. 
Suppose $\ell=1$. 
Let $\vv$ be an eigenvector of eigenvalue $\lambda$. Let $x_1, x_2, \dots, x_{d-1}$ be its entries  of the leaves and let $y$ be the entry of the root. Then $\lambda x_i = y$ for each $1 \leq i \leq d-1$. Since $\lambda \neq 0$, it forces $x_1 = \dots = x_{d-1} = x = a_0x$ for some $x$. Furthermore, $y = \lambda x = a_1 x$ as $a_1(\lambda) = \lambda$. The base case is proved. 

Suppose the claim holds for trees of at most $\ell$ levels.  We now show it holds for trees with $\ell+1$ levels. 
Let $z$ be the entry of the root $v$ in the eigenvector, and $y_1, \dots, y_{d-1}$ be the entries of its children, which are vertices $u_1, \dots, u_{d-1}$ respectively. For $1 \leq i \leq d-1$, let $T_i$ be the subtree of $T$ rooted at $u_i$. Thus $T_i$ is a $d$-ary tree with $\ell$ levels. By the inductive hypothesis, $\vv$ is radial on $T_i$. Let $x_i'$ be the entry of the leaves in $T_i$. By the inductive hypothesis, $y_i = a_{\ell-1} x_i'$. Furthermore, each of the $d-1$ children of $u_i$ has entry $y_i' = a_{\ell-2} x_i'$. On the other hand, we also have 
\[
\lambda y_i = (d-1) y_i' + z \implies \lambda a_{\ell-1} x_i' = (d-1) a_{\ell-2} x_i' + z.
\]
This implies $(\lambda a_{\ell-1} - (d-1)a_{\ell-2}) x_i' = z$, which is equivalent to $a_{\ell} x_i' = z$. Since $\lambda > 2\sqrt{d-1}$, we have  $a_\ell \neq 0$. 
Thus we conclude $x_1' = x_2' = \dots = x_{d-1}' = x$ for some $x$, and the root $v$ has value $z = a_\ell x$. The entries on the other levels are proved by the inductive hypothesis and the fact $x_1' = x_2' = \dots = x_{d-1}' = x$. 
\end{proof}

The next claim shows an explicit relationship between the eigenvalues and eigenvectors of $G'$ and  those of $T_s^\ell(d)G'$. 

\begin{claim}\label{lem:eigen-relation}
Let $d \geq 3, s,\ell \geq 2$ be positive integers and  
 $G'$ a fixed graph.
%Suppose $s < d$ are positive integers and $\ell\geq 2$ is an integer. 
 %with maximum degree at most $d$. 
 Let $G$ be $T_s^\ell(d) G'$. %a graph containing $G'$ where each vertex in $G'$ grows out $s$ number of $d$-ary trees with $\ell$ levels. %Then all the vertices in $G$ have degree $d$ except the leaves of the trees grown out. 
Then for any $\lambda > 2\sqrt{d-1}$, the following two statements are equivalent:
\begin{enumerate}
    \item $\lambda$ is an eigenvalue of $G$; 
    \item Let $a_i = a_i(\lambda)$ be defined in (\ref{eq:ai}). The following value $\mu$ is an eigenvalue of $G'$: \begin{equation}
        \lambda - {sa_{\ell-1}(\lambda)}/{a_{\ell}(\lambda)} = \mu.  \label{eq:lambdamu}
    \end{equation} 
\end{enumerate} %is an eigenvalue of $G$ with eigenvector $\vv'$ if and only if \[ \lambda - \frac{sa_{\ell-1}}{a_{\ell}} = \mu\] is an eigenvalue of $G'$. Here $a_i$'s are defined in (\ref{eq:ai}). 

In addition, suppose $\vv'$ is an  eigenvector of $G$ corresponding to $\lambda$. Then $\vv'$ restricted to $V(G')$ is an eigenvector $\vv$ of $G'$.  For each vertex $u$ that is on a $d$-ary tree joined to some vertex $v\in V(G')$ and is at a distance $i$ from the leaf of that tree, the value of $\vv'$ at $u$ is $\vv(v) a_i/a_\ell$.
\end{claim}
\begin{proof}
Let $v$ be a vertex in $G'$, and it joins $s$ trees $T_1, \dots, T_{s}$ of $\ell$ levels each, where the roots are $u_1, \dots, u_{s}$ respectively. To show Statement 1 implies 2, we can apply Claim \ref{claim:treeentry} to each of $T_i$. %Let $\vv'$ be the eigenvector of $G$ corresponding to $\lambda$. 
Suppose $\vv'$ restricted to the leaves of $T_i$ have entries $x_i$ (this is well-defined since the eigenvector is radial on $T_i$ by Claim \ref{claim:treeentry}).  Then each root $u_i$ has value $a_{\ell-1}x_i$.  Similarly, the  values of the $d-1$ children of $u_i$ in $T_i$ are $a_{\ell-2}x_i$. Then since $\lambda$ is an eigenvalue, we have 
\[
\lambda  a_{\ell-1}x_i = (d-1) a_{\ell-2} x_i + \vv'(v).
\]
%where $z$ is $\vv'(v)$.
Therefore  by a similar argument as in the proof of Claim \ref{claim:treeentry} we have $\vv'(v) = a_{\ell} x_i$. Since $a_\ell \neq 0$, we have $x_1 = x_2 = \dots = x_{s}=x$ for some $x$.  Thus $x = \vv'(v)/a_\ell$. 

We also have 
$
\lambda \vv'(v) =  \sum_{u:(u,v) \in E(G')} \vv'(u) + s a_{\ell-1}x
$
where the summation is over all the vertices adjacent to $v$ in $G'$.  Since $x = \vv'(v)/a_\ell$,  
\begin{align*}
\lambda \vv'(v)=\sum_{u : (u,v) \in E(G')} \vv'(u)+ s a_{\ell-1}\vv'(v)/a_\ell 
  \implies (\lambda - sa_{\ell-1}/a_\ell) \vv'(v) = \sum_{u :(u,v) \in E(G')} \vv'(u).
\end{align*}
Therefore $\lambda - sa_{\ell-1}/a_\ell$ is an eigenvalue of $G'$, with eigenvector $\vv'$  
 restricted to $V(G')$. The fact that Statement 2 implies Statement 1 is by the  same argument, constructing $\vv'$ directly from $\vv$ and noticing that $\lambda$ is an eigenvalue  corresponding to $\vv'$.  
\end{proof}

\vspace{0.2cm}

\begin{proof}[Proof of Lemma \ref{lem:gap}]
%The last assertion is proved by Lemma \ref{lem:G} already. We are left to prove that there is a unique eigenvalue larger than $2\sqrt{d-1}$.

%emma \ref{lem:G} asserts there is at least one eigenvalue larger than $2\sqrt{d-1}$, which is the solution to the equation 

The case when $s = 0$ is trivial. Thus we assume $s \geq 1$. 
Proving that there is at most one solution to (\ref{eq:d'}) that is larger than $2\sqrt{d-1}$ directly from the equation is challenging. Nevertheless, we can establish this fact by examining the eigenvector. As per Claim \ref{lem:eigen-relation}, the eigenvector of $\lambda$ in $G$ that is restricted to $V(G')$ is also the unique eigenvector $\vv$ of $G'$ for the top eigenvalue $\mu_1$, with each entry being non-negative.
For each vertex which is on the $d$-ary tree grew out from vertex $v \in V(G')$, its value on the eigenvector is $\vv(v) a_i/a_\ell \geq 0$ for some $0 \leq i \leq \ell-1$. 
 Suppose $\lambda$ and $\lambda'$ are two distinct solutions to (\ref{eq:d'}) that are larger than $2\sqrt{d-1}$. In that case, their eigenvectors are orthogonal but have all non-negative entries, a contradiction.

We now prove statements 1 and 2. Let $\lambda_0$ be the largest solution to (\ref{eq:d'}). By Claim \ref{lem:eigen-relation}, it suffices to show that for any $\mu < \mu_1$, the following equation in terms of $\lambda$ has no solution larger than $\max(2\sqrt{d-1}, \lambda_0)$.
\begin{equation}
   h(\lambda) :=  \lambda - sa_{
    \ell-1}/a_\ell = \mu. \label{eq:new1}
\end{equation} 
For the sake of contradiction, suppose $\lambda' > \max(2\sqrt{d-1}, \lambda_0)$ is a solution to (\ref{eq:new1}) for some $\mu < \mu_1$, i.e., $h(\lambda') = \mu < \mu_1$. 
Note $a_\ell, a_{\ell-1}$ are also functions of $\lambda$. When $\lambda > 2\sqrt{d-1}$, $a_{\ell-1}/a_\ell < \frac{2}{\lambda+\sqrt{\lambda^2 - 4(d-1)}}$. 
Therefore $h(\lambda) \to \infty$ if $\lambda \to \infty$. Combining with the fact that $h(\lambda') = \mu < \mu_1$ and by the continuity of $h$,  equation (\ref{eq:d'}) that $h(\lambda) = \mu_1$ has a solution $\lambda''  \geq \lambda'$. Since $\lambda' > \max(2\sqrt{d-1}, \lambda_0)$, the existence of such a $\lambda''$ contradicts with the fact that $\lambda_0$ is the largest solution to (\ref{eq:d'}).

We now prove the last statement. Let $\mu$ be any eigenvalue of $G'$ which is less than $2\sqrt{d-1+\epsilon} - s/(d-1+\epsilon)$. %It suffices to show that  there is no solution larger than $2\sqrt{d-1+\epsilon}$ to  (\ref{eq:new1}). 
%Suppose there is a solution $\lambda > 2\sqrt{d-1}$ to (\ref{eq:new1}) as otherwise we are done. 
Thus for any $\lambda \geq 2\sqrt{d-1+\epsilon}$ and any $\ell$ and $\epsilon \geq 0$, 
\begin{align*}
  h(\lambda) \geq  \lambda - 2s/ (\lambda+\sqrt{\lambda^2 - 4(d-1)}) \geq  2\sqrt{d-1+\epsilon} - s/(\sqrt{d-1+\epsilon}  + \sqrt{\epsilon})> \mu.
\end{align*}
%Thus if $\mu \leq 2\sqrt{d-1+\epsilon} - s/\sqrt{d-1+\epsilon}$, then
Thus (\ref{eq:new1}) has no solution at least $2\sqrt{d-1+\epsilon}$, as desired. 
\end{proof}

\vspace{0.2cm}
We are now ready to prove the main ingredient: Lemma \ref{lem:maingadget}.

\begin{proof}[Proof of Lemma \ref{lem:maingadget}]
Let $G_0'$ be an $N$-lift of the complete graph $K_{d+1}$ for sufficiently large $N$.  For each vertex $i$ in $K_{d+1}$ where $1 \leq i \leq d+1$, let $V_i$ be the set of vertices in the lift which are the pre-images of $i$ through the covering map. Suppose $G_0'$ has girth $L \geq 100d/\epsilon$ and $\lambda_2(G_0') \leq 2\sqrt{d-1} + \epsilon/2$. Such a graph exists by the work of Bordenave and Collins \cite{Bo2}. Since a random $N$-lift of $K_t$ for $t \geq 4$ has the second largest eigenvalue at most $2\sqrt{t-2} + \epsilon/2$ with probability at least $1 - O(N^{-0.99})$ \cite{Bo2}, we may further assume by a union bound that for each $1 \leq i \leq d-2$,
\begin{equation}
    \lambda_2(G_0'[V_i \cup V_{i+1} \cup \dots \cup V_{d+1}]) < 2\sqrt{d-i} + \epsilon/2. \label{eq:expand}
\end{equation}
By an abuse of notation, label the vertices in $G_0'$ as $1, 2, \dots, |V(G_0')| = (d+1)N$ such that vertices in $V_i$ comes before vertices in $V_{i+1}$ for each $1 \leq i \leq d$. 
For each $1 \leq i \leq |V(G_0')|$, let $G_i'$ be 
the graph obtained from $G_{i-1}'$ by removing vertex $i$ from $V(G_0')$. Clearly for each $1 \leq t \leq d$, the graph $G_{tN}'$ is $(d-t)$-regular since by removing  vertices in $V_1\cup \dots \cup V_t$ from $G_0'$, the remaining graph is an $N$-lift of $K_{d-t+1}$.

Let $I = \lceil 10\sqrt{d-1}/\epsilon \rceil$. 
In a way similar to the previous argument,
the interpolation procedure begins with the graph $G_0 = T^I (d) G_0'$. For each $0 \leq i \leq (d+1)N$, let $G_i = T^I(d) G_i'$. The procedure stops as soon as $i = dN$ or when the top eigenvalue of $G_i$ is at most $2\sqrt{d-1}+\epsilon/2$. %We will show one of the graphs $G_i$ have all the desired properties required by Lemma \ref{lem:maingadget}. 

\begin{claim}
\label{claim:smallstep}
    For each $0 \leq i \leq (d+1)N$, if $\lambda_1(G_i) > 2\sqrt{d-1} + \epsilon/10$, then $|\lambda_1(G_i) - \lambda_1(G_{i+1}) | \leq 6d/L$. 
\end{claim}
\begin{proof}
Note that by construction, $G_{i+1}$ is obtained from $G_i$ by removing 
vertex $i+1$, omitting several  connected components each of which 
is a $d$-ary tree, and adding at most $d$ such components, 
joining their roots to the graph. 
Indeed by removing the vertex $i+1$ from $G_i$, there are several connected
components each of which is a $d$-ary tree, connected in $G_i$
by an edge from the root to the vertex
$i+1 \in V(G_i)$. Let $G_i''$ be the
induced subgraph of $G_i$ obtained 
by removing these trees. Therefore $G_i$ is
the graph obtained from $G_i''$ and several $d$-ary trees so that for each
such a tree, there is one edge between its root and 
the vertex $i+1$ in $G_i''$.
Similarly, $G_{i+1}$ can be considered as starting from the graph $G_i''
\setminus \{i+1\}$, and then for each vertex $u \in V(G_i'')$ which is a
neighbor of vertex $i+1$ in $G_i''$, this vertex $u$ is joined 
in $G_{i+1}$ to the root of a new copy of a $d$-ary tree.

 The largest eigenvalue of $G_i'' \setminus \{i+1\}$ is at most $\lambda_1(G_i)$ by eigenvalue interlacing. The largest eigenvalue of each 
tree is at most $2\sqrt{d-1}$. Applying Lemma \ref{c28} $d$ times, we have $\lambda_1(G_{i+1}) \leq \lambda_1(G_i) + 4d/L$.

For the other direction, %again let the vertex $u$ be the vertex removed from $G_i'$ to obtain $G_{i+1}'$. Then $G_i \setminus \{u\}$ is at most $d$ vertex disjoint $d$-ary trees $T$. Let the induced subgraph of $G_i$ on the vertices not on these trees be $G_i''$. Thus these trees join $G_i''$ at $u$ by edges. 
by applying Lemma \ref{c28} at most $d$ times to $T \cup G_i''$ and $G_i$,  \begin{equation}
   | \lambda_1(G_i) - \max(\lambda_1(T), \lambda_1(G_i''))| \leq 4d/L. \label{eq:Gi''}
\end{equation} 
If $\lambda_1(G_i'') \leq 2\sqrt{d-1}$, then since the eigenvalues of $T$ are also no more than $2\sqrt{d-1}$, (\ref{eq:Gi''}) implies $\lambda_1(G_i) \leq  2\sqrt{d-1} + 4d/L < 2\sqrt{d-1} + \epsilon/10$, a contradiction.  
Therefore we can assume $\lambda_1(G_i'') > 2\sqrt{d-1}$, and thus is larger than $\lambda_1(T)$. 
This fact together with (\ref{eq:Gi''}) imply $\lambda_1(G_i'') \geq \lambda_1(G_i) - 4d/L$. 
Let $A, A''$ be the adjacency matrices of $G_{i+1}, G_{i}''$ respectively, and by adding zero entries to vertices in $V(G_{i}'') \setminus V(G_i+1)$ and $V(G_{i+1}) \setminus V(G_i'')$ respectively. Let $\vv, \vv''$ be the top normal eigenvectors of $G_{i+1}, G_{i}''$ respectively. Thus $\lambda_1(G_{i+1}) - \lambda_1(G_i'') \geq \vv''^T (A - A'') \vv''$, which effectively is a sum of at most $d$ terms of the form $\pm \vv''(u)\vv''(i+1)$ for several different vertices $u$'s. By Claim \ref{cl22}, each such a term has absolute value at most $2/(L+1)$. Thus $\lambda_1(G_{i+1}) - \lambda_1(G_i'') \geq -2d/(L+1).$ Combining with the lower bound on $\lambda_1(G_i'')$,  it follows that $\lambda_1(G_{i+1}) - \lambda_1(G_i) \geq -6d/L.$
\end{proof}
\vspace{0.2cm}

As $\lambda_1(G_0) = d$ and $G_{dN}$ is the disjoint union of trees which has top eigenvalue at most $2\sqrt{d-1}$, by Claim \ref{claim:smallstep}, 
there is an $i^*$ such that $|\lambda_1(G_{i^*})-z| < \epsilon/2$. Furthermore, by Claim \ref{claim:notgrow}, for any $\ell \geq \lceil 10\sqrt{d-1}/\epsilon \rceil$, $|\lambda_1(T^\ell(d) G_{i^*}' ) -  \lambda_1(T^I(d)G_{i^*}') | \leq \epsilon/2$. Therefore for any $\ell \geq \lceil 10\sqrt{d-1}/\epsilon \rceil$, $|\lambda_1(T^\ell(d) G_{i^*}' ) -  z| \leq \epsilon$, as desired. 

It remains to prove Statement 2. Fix any $\ell$. 
Suppose $G_{i^*}'$ has maximum degree at most $t+1$ but has some vertex of degree $t$. Here $0 \leq t \leq d-1$. Thus $G_{i^*}'$ is an induced subgraph of the $(t+1)$-regular graph $G_0'[ V_{d-t} \cup\dots\cup V_{d+1}]$ by the construction. Passing down to the connected component of $G_{i^*}$ with the largest top eigenvalue, call its $\tilde G_{i^*}$. 
This connected component is an induced subgraph of some connected component of $T_{d-t}^\ell(d) G_0'[V_{d-t}\cup \dots \cup V_{d+1}]$. If $t \geq 2$, 
then by (\ref{eq:expand}), $\lambda_2(G_0'[ V_{d-t} \cup\dots\cup V_{d+1}]) \leq 2\sqrt{t} + \epsilon/2$. 
When $t\leq 1$, then each connected component of $G_0'[ V_{d-t} \cup\dots\cup V_{d+1}]$ is a cycle or an edge, thus has second largest eigenvalue 
smaller than $2\sqrt{t}$. If $\eps'$ is chosen such that $2\sqrt{d-1+\eps'} - (d-t)/\sqrt{d-1+\eps'} \geq 2\sqrt{t} + \epsilon/2$, 
then by Lemma \ref{lem:gap}, each connected component of $T_{d-t}^\ell(d) G_0'[V_{d-t}\cup \dots \cup V_{d+1}]$ has second largest eigenvalue at most $2\sqrt{d-1+\eps'}$. The inequality holds for all $t \leq d-1$ if $\eps'$ is chosen such that $2\sqrt{d-1+\eps'} \geq  1/\sqrt{d-1} + 2\sqrt{d-1} + \epsilon/2$. Thus by eigenvalue interlacing, 
$\lambda_2(T^\ell(d)\tilde G_{i^*}) \leq 2\sqrt{d-1} + \frac{1}{\sqrt{d-1}} +\epsilon/2$, as desired. 

The required divisibility of the number of vertices of $T^\ell(d)\tilde G$ could be shown through the same interpolation analysis again, by noticing that within every $d$ steps, there must be one step where the number of vertices is divisible by $d$. 
\end{proof}

\section{Near Ramanujan graphs with localized eigenvectors}

The proof of Theorem \ref{thm:vectormain} is similar to the one described 
in the previous section, together with an additional 
quantitative analysis that establishes the eigenvector  
localization. The gadget to be applied to the patching 
lemma (Lemma \ref{lem:patchingshort}) is as follows. 

\begin{lemma}\label{lem:smalltopeigen}
Fix integers $n_0, d \geq 3$. Let $\epsilon_1, \epsilon_2>0$ be sufficiently small in terms of $d$. There is a graph $\tilde G$ on at least 
$n_0$ vertices with maximum degree at most $d$, girth at least 
$0.5 \log_d |V(\tilde G)|$, such that for any 
$\ell \geq \lceil 100\sqrt{d-1}/\min(\epsilon_1, \epsilon_2) \rceil$, 
the largest eigenvalue of $T^\ell(d) \tilde G$ is in the interval 
$(2\sqrt{d-1}+\epsilon_1, 2\sqrt{d-1}+\epsilon_1 + \epsilon_2)$. 
We may also assume the number of vertices of 
degree one in $T^\ell(d) \tilde G$ is divisible by $d$. 
\end{lemma}
\begin{proof}[Proof of Lemma \ref{lem:smalltopeigen}]
Let $G_0'$ be a $d$-regular bipartite graph on $n \geq 2n_0$ vertices 
which has girth at least $0.5 \log_{d} n$. 
Such graphs exist by \cite{ES} (and explicit constructions are given in
\cite{LPS} where one can omit some of the generators to get the required
degree if it is not of the form given in \cite{LPS}). 
Since $G_0'$ is bipartite, it contains an independent set $U$ with 
$|U| \geq n/2$. Fix such a $U$. Label the vertices of $G_0'$ 
from $1$ to $n$ such that the vertices in $U$ are ranked the last. 
Define a sequence of graphs $(G_i')$ where for each $1 \leq i \leq n$, 
$G_i'$ is obtained from $G_{i-1}'$ by removing vertex $i$. 
Let $\epsilon = \min(\epsilon_1, \epsilon_2)$ 
and $I =  \lceil 100\sqrt{d-1}/\epsilon \rceil$. 
Define a new graph sequence $(G_i)$ where for 
each $0 \leq i \leq n$, $G_i = T^I(d) G_i'$. 
By a similar analysis to the one in the proof of  
Claim \ref{claim:smallstep}, when $n$ is sufficiently large, 
$|\lambda_1(G_i)  - \lambda_1(G_{i+1})| \leq 0.01\epsilon$ 
as long as $\lambda_1(G_i) \geq 2\sqrt{d-1} + 0.5\epsilon_1$. 
Note that $\lambda_1(G_0) = d$. On the other hand, 
$\lambda_1(G_{n-|U|}) < 2\sqrt{d-1}$ since $G_{n-|U|}'$ is 
the independent set $U$ and thus $G_{n-|U|}'$ is a disjoint 
union of $d$-ary trees. Thus there is an $1 \leq i \leq n - |U|$ 
such that  $\lambda_1(G_i) \in (2\sqrt{d-1} + \epsilon_1 
+ 0.5\epsilon, 2\sqrt{d-1} + \epsilon_1 + 0.6\epsilon)$.  
Let $\tilde G$ be $G_i'$. 
As in Claim \ref{claim:notgrow}, 
for any $\ell \geq \lceil 100\sqrt{d-1}/\epsilon\rceil$, 
$|\lambda_1(T^\ell(d) G_{i}' ) -  \lambda_1(G_i) | \leq \epsilon/20. $
Thus the desired bound on $\lambda_1(T^\ell(d) \tilde G)$ holds. 
Note that by the construction, $\tilde G = G_i'$ 
has at least $|U| \geq n/2 \geq n_0$ 
vertices, and the girth of $\tilde G$ is at least $0.5 
\log_d |V(\tilde G)|$. 
\end{proof}
\vspace{0.2cm}

\begin{proof}[Proof of Theorem \ref{thm:vectormain}]
Assume $\beta > 0$ is  sufficiently small  in terms of $d$, $\ell$ 
is large in terms of $d, \beta$, and $n_0$ is large in terms 
of $d, \ell, \beta$. 
By Lemma \ref{lem:smalltopeigen},
there is a graph $\tilde G$ on $n \geq n_0$ vertices with maximum degree 
at most $d$, girth at least $0.5 \log_d |V(\tilde G)|$ and   
$\lambda_1(T^\ell(d) \tilde G) = \mu_1 \in (2\sqrt{d-1}+0.5\beta, 2\sqrt{d-1}+0.8\beta)$. 

Let $F_1$ be $T^\ell(d) \tilde G$. Note that $F_1$ has at most $d^{\ell+1}n$ vertices. 
Let $F_0$ be a $d$-regular Ramanujan graph on $m$ vertices with girth at least $2
\log m/3$ and $m = n^C$ where $C\geq 4$. We can assume $d^{\ell+1}n \leq n^{\sqrt{C}}$. The existence of such a graph and the fact that it can be constructed explicitly when $d = p+1$ where $p$ is a prime is a result 
of Lubotzky, Phillips and Sarnak \cite{LPS}, and Margulis \cite{Ma}. 
Applying the patching Lemma \ref{lem:patchingshort} to $F_0$ and $F_1$, we obtain a graph $F$ where $\lambda_2(F) \in (2\sqrt{d-1}+0.4\beta, 2\sqrt{d-1} + 0.9\beta)$. 

Let $\vv$ be the normal eigenvector of $F$ corresponding to the 
second largest eigenvalue. We proceed to show that $\vv$ is localized.
Let $X$ be the union of the  leaves in $F_1$ and the set of vertices in $V(F) \setminus V(F_1)$ adjacent to these leaves. 

Let $A_C$ be the adjacency matrix for the induced subgraph of $F$ on $X$, 
which is a disjoint union of stars, each having $d-1$ leaves.  
Let $F_b$ be the big subgraph of $F$ induced on $V(F) \setminus  V(F_1)$ and let $A_b$ be its adjacency matrix. Let  $A_s$ be the adjacency matrix of the relatively small graph $F_1$. 
Assume all those adjacency matrices are of dimension 
$|V(F)| \times |V(F)|$ by filling zeros in the additional columns and rows. 
Then $\vv^t A_F \vv= \vv^t A_C \vv+ \vv^t A_s \vv + \vv^t A_b \vv $. 
We bound each of these three terms separately. 

The first term satisfies $\vv^t A_C \vv %= & \sum_{u \in X_{\ell}} \sum_{u' \in X_{\ell+1}, u'\sim u} \vv(u) v(u')  \\
    %\leq & \sum_{u \in X_{\ell}} \sum_{u' \in X_{\ell+1}, u'\sim u} 0.5 \vv(u)^2  + 0.5 \vv(u')^2 \\
   % \leq & 0.5({d-1}) \sum_{u \in X_{\ell} \cup X_{\ell+1}} \vv(u)^2  
   \leq \sqrt{d-1} \sum_{u\in X} \vv(u)^2 \leq  32(\sqrt{d-1})/\ell$, %\label{eq:small1}
%\end{align}
where the last inequality is by Lemma \ref{lem:treemiddlelayer}. The second term satisfies   %by Lemma \ref{lem:G}, 
%\begin{equation}
   $ \vv^t A_s \vv \leq \mu_1\sum_{u \in V(F_1)}\vv(u)^2.$ % (2+\beta)\sqrt{d-1} \sum_{u \in V(F_s)}\vv(u)^2. 
  %  \label{eq:main1}
%\end{equation}
%This term is very small when fixing $\|\vv\|_2^2$. 
It remains to bound the last term $\vv^t A_b \vv$. Let $\vv'$ be a vector 
indexed by $V(F_0)$ which is equal to $\vv$ on $V(F_b)$ and $0$ elsewhere on $V(F_0)$. Write $\vv' = c_1 1 + c_2f$ where $f$ is orthogonal to $1$ and has norm one. Then $|c_1|= |\vv'^t 1/ |V(F_0)|| $. Since  $\vv$ is orthogonal to $1$, $|\vv'^t1| = |\sum_{u \in V(F_b)} \vv(u)|=|\sum_{u \in V(F_1)} \vv(u)|$,  and thus $|c_1| \leq \sqrt{|V(F_1)|}/|V(F_0)|$ by Cauchy-Schwarz. Let $A_0$ be the adjacency matrix of $F_0$. Then  $\vv^t A_b \vv =   \vv'^t A_0 \vv' = c_1^2 1^t A_0 1 + c_2^2 f^t A_0 f$. Thus 
\begin{align*}
  \vv^t A_b \vv \leq (|V(F_1)|/|V(F_0)|^2) d |V(F_0)| + 2\sqrt{d-1} \|\vv'\|_2^2 
    \leq   %d \|\vv\|_2^2/m + 2\sqrt{d-1}   \|\vv'\|_2^2  = 
    \frac{d|V(F_1)|}{|V(F_0)|} + 2\sqrt{d-1} \sum\nolimits_{u \in V(F_b)}\vv(u)^2. %\label{eq:small2}
\end{align*}
Here the inequalities are by the  bounds on $|c_1|, |c_2|$ and the fact that $\lambda_2(F_0)=2\sqrt{d-1}$. %The last inequality is by the fact that $|V(G)| \leq n d^{\ell+1} \leq \sqrt{m}$. 

Adding the upped bounds on the three terms, 
\begin{align}
   & \vv^t A_F \vv \leq 
    \frac{32}{\ell}\sqrt{d-1} +  \frac{d|V(F_1)|}{|V(F_0)|}  + 2\sqrt{d-1}\sum\nolimits_{u \in V(F_b)}\vv(u)^2 + \vv^t A_s\vv. \label{eq:vAFv0}
    \\
   \leq & \frac{32}{\ell} \sqrt{d-1}+  \frac{d|V(F_1)|}{|V(F_0)|}  
+ 2\sqrt{d-1}\sum\nolimits_{u \in V(F_b)}\vv(u)^2 + \mu_1 \sum\nolimits_{u \in V(F_1)}\vv(u)^2. \label{eq:vAFv}
\end{align}
It is not difficult to see that by eigenvalue interlacing, 
$\mu_1 \leq \vv^t A_F \vv.$
This together with (\ref{eq:vAFv}) implies 
\[
\mu_1 \leq \frac{32}{\ell} \sqrt{d-1} +  \frac{d|V(F_1)|}{|V(F_0)|} 
+  \mu_1 \sum\nolimits_{u \in V(F_1)}\vv(u)^2 + 2\sqrt{d-1} \sum\nolimits_{u \in V(F_b)}\vv(u)^2
\]
Subtracting from both sides $2\sqrt{d-1}  = 2\sqrt{d-1}
(\sum_{u \in V(F_1)}\vv(u)^2 + \sum_{u \in V(F_b)}\vv(u)^2)$, we get
\[
\left(\mu_1 - 2\sqrt{d-1} -\frac{32}{\ell}\sqrt{d-1}-  \frac{d|V(F_1)|}{|V(F_0)|}
\right) \leq (\mu_1-2\sqrt{d-1}) \sum\nolimits_{u \in V(F_1)}\vv(u)^2 .
\]
Since $\mu_1 \geq 2\sqrt{d-1} +0.5 \beta$ and 
$\frac{32}{\ell}\sqrt{d-1}+  \frac{d|V(F_1)|}{|V(F_0)|} 
\leq \frac{32}{\ell}\sqrt{d-1}+\frac{d}{n}  
\leq 0.5 \beta^2$ %100 \sqrt{C}/ \log_d m$, 
 we have that 
 \[\sum\nolimits_{u \in V(F_1)}\vv(u)^2 /\|\vv\|_2^2 \geq 
1-\frac{0.5 \beta^2}{0.5 \beta} = 1 - \beta. \]
The desired result follows.
\end{proof}

\section{Remarks}
\begin{itemize}
\item
The quantitative estimates in Theorem \ref{t21} can be improved
for values of $n$ and $d$ for which it is known that there are
high girth Ramanujan graphs. In particular, by the constructions of
Lubotzky, Phillips and Sarnak \cite{LPS} and Margulis  \cite{Ma}
for every  $d=p+1$ with $p$ a prime congruent to $1$ modulo $4$,
there are infinitely many values of $n$ for which there are 
explicit $d$-regular Ramanujan graphs on $n$ vertices with girth
$\Omega(\log n/ \log d)$. Plugging such a graph as  $G_0$ in the proof
we get the assertion of Theorem \ref{t21} in which $\log \log n$
is replaced by $\log n$. A similar remark applies to the proofs
of Theorem \ref{t12} and Theorem \ref{t25}. 
\item
Conjecture \ref{c13} remains open. It is true, however, that if
there is a $d$-regular graph $H$ with top $k$ eigenvalues
$d=\mu_1  > \mu_2 \geq \ldots \geq \mu_k >2 \sqrt{d-1}$
then there are infinitely many connected $d$-regular graphs
with the same sequence of $k$ top eigenvalues. This follows from
the result of Friedman and Kohler \cite{FK}, see also \cite{Bo},
that all the new eigenvalues of random lifts of $H$ are, with
high probability, at most $2\sqrt{d-1}+o(1)$ where the $o(1)$-term
tends to zero as the size of the lift tends to infinity. 
\item
The problem of
understanding the possible spectrum of
finite $d$-regular graphs is challenging. Some aspects of this
problem are considered here, other variants appear in
\cite{KS}, \cite{Yu}.
\item
Theorem \ref{thm:vectormain} can be extended to yield near-Ramanujan
regular graphs with  multiple localized eigenvectors corresponding 
to eigenvalues strictly larger than $2\sqrt{d-1}$. This can be proved
in a similar way, 
by patching multiple graphs. The detailed proof requires a more technical computation, and we thus decided not to include it here. The proof can be found on the second author's homepage \cite{comp}. 
\end{itemize}
\vspace{0.2cm}

\noindent
{\bf Acknowledgment:}\, We thank Peter Sarnak for suggesting the 
first problem considered here and also thank Vishal Gupta,
Jiaoyang Huang, Nike Sun and Michael Ren  for helpful discussions and comments.

\end{document}